\documentclass[reqno]{amsart}

\usepackage[sorting=none, sorting=nyt, maxbibnames=99, backend=biber]{biblatex}

\renewbibmacro{in:}{}
\bibliography{references.bib}
\usepackage{amssymb}
\usepackage{calrsfs}
\usepackage{graphics,graphicx,mathrsfs}
\usepackage{enumerate}
\usepackage{enumitem}
\usepackage{textcomp}
\usepackage{bbold}
\usepackage{url}
\usepackage{xcolor}
\definecolor{vio}{rgb}{0.54, 0.17, 0.89}
\usepackage[ruled, lined, linesnumbered, longend]{algorithm2e}
\usepackage{hyperref}
\usepackage{titlesec}
\usepackage{todonotes}

\newtheorem{theorem}{Theorem}[section]
\newtheorem{lemma}[theorem]{Lemma}

\newtheorem{proposition}[theorem]{Proposition}
\newtheorem{conjecture}[theorem]{Conjecture}

\numberwithin{equation}{section}

\theoremstyle{remark}
\newtheorem*{remark}{Remark}

\theoremstyle{definition}
\newtheorem{definition}{Definition}
\newtheorem{hypothesis}[theorem]{Hypothesis}

\titleformat{\section}
  {\normalfont\large\bfseries\centering}{\thesection}{1em}{}

\titleformat{\subsection}
  {\normalfont\bfseries}{\thesubsection}{1em}{}  

\titleformat{\subsubsection}
  {\normalfont\itshape}{\thesubsubsection}{1em}{}

\DeclareMathOperator{\Z}{\mathbb{Z}}

\DeclareMathOperator{\R}{\mathbb{R}}
\DeclareMathOperator{\N}{\mathbb{N}}

\def\cA{{\mathcal A}}
\def\cB{{\mathcal B}}

\def\cR{{\mathcal R}}
\def\cS{{\mathcal S}}

\def\fB{{\mathfrak B}}
\def\fm{{\mathfrak m}}
\def\fM{{\mathfrak M}}
\def\fS{{\mathfrak S}}

\def\E{{\mathbb E}}

\newcommand{\trev}{\text{rev}}

\def\reals{\hbox{\rm I\kern-.18em R}}
\def\complexes{\hbox{\rm C\kern-.43em
\vrule depth 0ex height 1.4ex width .05em\kern.41em}}
\def\field{\hbox{\rm I\kern-.18em F}} 

\let\svthefootnote\thefootnote
\newcommand\freefootnote[1]{%
  \let\thefootnote\relax%
  \footnotetext{#1}%
  \let\thefootnote\svthefootnote%
}

\newenvironment{section*}[2][A]{
  \section*{#2}
  \renewcommand\thesection{#1}
  \setcounter{theorem}{0}}{}

\newcommand{\rev}[1]{\overleftarrow{#1}}

\allowdisplaybreaks

\begin{document}

\title[The reverse Goldbach problem]{The reverse Goldbach problem and a refined Zsiflaw--Legeis theorem}

\author{Michael Harm}
\author{Daniel R. Johnston}
\thanks{The first author was supported by a University International Postgraduate Award (UIPA, RSRE7061, RSRE7063), a University of New South Wales School of Mathematics scholarship (RSRT6014) and an Australian Research Council Grant (ARC, RSGT002)}
\thanks{The second author was supported by Australian Research Council DP240100186.}
\address{School of Mathematics and Statistics, UNSW Sydney, Australia}
\email{m.harm@unsw.edu.au}
\address{School of Science, UNSW Canberra, Australia}
\email{daniel.johnston@unsw.edu.au}
\date\today

\begin{abstract}
    We prove new results on the additive theory of reversed primes $\rev{p}$; that is, primes $p$ which are written backwards in a fixed base $b\geq 2$. In particular, we study a variant of Goldbach's conjecture, looking at representations of integers as the sum of primes and reversed primes. We show that:
    \begin{enumerate}
        \item Every large odd integer is the sum of a prime and two reversed primes ($N=p_1+\rev{p_2}+\rev{p_3}$).
        \item Every large odd integer is the sum of two primes and a reversed prime ($N=p_1+p_2+\rev{p_3}$).
        \item Almost all even integers are the sum of a prime and a reversed prime ($N=p_1+\rev{p_2}$).
        \item All large integers are the sum of a reversed prime and a square-free number ($N=\rev{p}+\eta$, $\mu^2(\eta)=1$).
    \end{enumerate}
    To obtain our results, along with associated asymptotics, we apply the Hardy--Littlewood circle method and a novel refinement of the ``Zsiflaw--Legeis" theorem on the distribution of reversed primes in arithmetic progressions. Notably, our variant of the Zsiflaw--Legeis theorem does not require one to fix the digit length, unlike previous versions.
\end{abstract}

\maketitle 

\freefootnote{\textit{Corresponding author}: Daniel Johnston (daniel.johnston@unsw.edu.au).}
\freefootnote{\textit{Affiliation}: School of Science, The University of New South Wales Canberra, Australia.}
\freefootnote{\textit{Key phrases}: Reversed primes, Goldbach's conjecture, circle method, Siegel--Walfisz theorem.}
\freefootnote{\textit{2020 MSC codes}: 11A63, 11P32 (Primary) 11L07, 11N13, 11P55 (Secondary).}

\section{Introduction}
\subsection{Reversed primes and Hcabdlog's conjecture}
For a fixed base $b\geq 2$, a \emph{reversed prime} $\rev{p}$ is a prime $p$ written backwards in base $b$. For example, in base $10$,
\begin{equation*}
    32=\rev{23},\quad 145=\rev{541},\quad 7681=\rev{1867},
\end{equation*}
are reversed primes. Historically, reversed primes have been studied as part of recreational and elementary number theory. However, in recent years there have been several new studies on the analytic properties of reversed primes; see~\cite{dartyge2024reversible,bhowmik2024telhcirid,chourasiya2025power,santana2025,dartyge2025prime,bhowmik2025zsiflaw}. Here, we focus on the additive theory of reversed primes, proving the following four results, which are variations of Goldbach's conjecture.

\begin{theorem}\label{tern1thm:intro}
    In any fixed base $b\geq 2$, every sufficiently large odd integer $N$ can be expressed as the sum of a prime and two reversed primes:
    \begin{equation*}
        N=p_1+\rev{p_2}+\rev{p_3}.
    \end{equation*}
\end{theorem}

\begin{theorem}\label{tern2thm:intro}
    In any fixed base $b\geq 2$, every sufficiently large odd integer $N$ can be expressed as the sum of two primes and a reversed prime:
    \begin{equation*}
        N=p_1+p_2+\rev{p_3}.
    \end{equation*}
\end{theorem}

\begin{theorem}\label{almostthm:intro}
    For any fixed base $b\geq 2$ and $A>0$, all but $O_{b,A}(x/(\log x)^A)$ even integers $N\leq x$ can be expressed as the sum of a prime and a reversed prime:
    \begin{equation}\label{almostthmeq}
        N=p_1+\rev{p_2}.
    \end{equation}
    That is, almost all even $N$ can be expressed in the form~\eqref{almostthmeq}.   
\end{theorem}
\begin{theorem}\label{sqfreethm:intro}
    In any fixed base $b\geq 2$, every sufficiently large integer $N$ can be expressed as the sum of a reversed prime and a square-free number:
    \begin{equation*}
        N=\rev{p}+\eta,\qquad \eta\ \text{square-free.}
    \end{equation*}
\end{theorem}

Asymptotic versions of Theorems~\ref{tern1thm:intro},~\ref{tern2thm:intro} and \ref{sqfreethm:intro} are given later; see Theorems~\ref{ternaryhcabthm1},~\ref{ternaryhcabthm2} and~\ref{sqfreethm}. We remark that our results are inspired by other recent works combining additive number theory with the digital properties of integers. For example, studies on representing integers as the sum of palindromes~\cite{banks2016every,cilleruelo2018every,raj2020,zakharov2024}, the sum of primes with missing digits~\cite{leng2025}, or the sum of Niven numbers~\cite{sanna2021additive,thomas2025niven}.

To understand what further results should be true, we now state Hcabdlog's conjecture\footnote{Hcabdlog is Goldbach spelt backwards. A similar naming convention is used in other results throughout the literature and this paper.}, introduced in~\cite[\S 4.2]{chourasiya2025power}. Notably, we give a simplified variant of~\cite[Conjecture~4.2]{chourasiya2025power} upon restricting to reversed primes with ${\gcd(\rev{p},b^3-b)=1}$. Including this coprimality condition simplifies many arguments and is ubiquitously used throughout the literature. See for example the introduction of~\cite{bhowmik2024telhcirid}, where the relationship between $n$ and $\rev{n}$ mod $b^3-b$ is discussed in detail.
\begin{conjecture}[Hcabdlog's conjecture]\label{hcabcon}
    For any fixed base $b\geq 2$, every sufficiently large even number $N$ can be expressed as
    \begin{equation}\label{hcabeq}
        N=p_1+\rev{p_2},
    \end{equation}
    where $p_1$ and $p_2$ are prime, and $\gcd(\rev{p_2},b^3-b)=1$.
\end{conjecture}
Conjecture~\ref{hcabcon} thus suggests that the exceptional set in Theorem~\ref{almostthm:intro} should in fact be $O_b(1)$. Note that the restriction to even $N$ is required since $b^3-b$ is always even so that $\rev{p_2}$ is always odd. This also explains the parity restrictions in Theorems~\ref{tern1thm:intro},~\ref{tern2thm:intro} and~\ref{almostthm:intro}. Later, we also give an asymptotic version of Hcabdlog's conjecture based on circle method heuristics (Conjecture~\ref{asymhcabcon}).

It is also natural to wonder whether every large even $N$ can be expressed as the sum of two reversed primes, that is
\begin{equation}\label{purehcabdlogeq}
    N=\rev{p_1}+\rev{p_2}.
\end{equation}
However, \eqref{purehcabdlogeq} fails infinitely often for infinitely many bases $b\geq 2$, including the ``human base" $b=10$. The reason for this is that the leading (i.e.~most significant) digit of a reversed prime is coprime to $b$. This leads to large gaps between reversed primes that can make~\eqref{purehcabdlogeq} impossible to satisfy. However, it seems that~\eqref{purehcabdlogeq} should hold when $b$ is a prime, and that a ternary version should hold in most bases, including $b=10$.
\begin{conjecture}\label{primebasecon}
    Let $b$ be a fixed prime. Then every sufficiently large even integer $N$ can be expressed as the sum of two reversed primes:
    \begin{equation*}
        N=\rev{p_1}+\rev{p_2}.
    \end{equation*}
\end{conjecture}
\begin{conjecture}\label{b10con}
    Let $b=10$. Then every sufficiently large odd integer $N$ can be expressed as the sum of three reversed primes:
    \begin{equation*}
        N=\rev{p_1}+\rev{p_2}+\rev{p_3}.
    \end{equation*}
\end{conjecture}

Further discussion on these conjectures is relegated to the final section of the paper (Section~\ref{sumsect}). In what follows, we instead focus on the representation of integers as a \emph{mix} of primes and reversed primes.

In the following subsections we state our main technical results. Firstly, in Section~\ref{ZLsect} we introduce additional notation and give a new asymptotic result for reversed primes in arithmetic progressions. This asymptotic offers additional flexibility compared to previous such results in the literature. In Section~\ref{Appsect}, we then state asymptotic versions of Theorems~\ref{tern1thm:intro},~\ref{tern2thm:intro} and~\ref{sqfreethm:intro}, and Conjecture~\ref{hcabcon}. An outline of the rest of the paper is then given in Section~\ref{outlinesect}.

\subsection{Further notation and a refined Zsiflaw--Legeis theorem}\label{ZLsect}
To make our definitions more precise we define
\begin{equation*}
    \mathcal{B}_L:=\{n\in\mathbb{N}:b^{L-1}\leq n<b^L\}
\end{equation*}
to be the set of positive integers with $L$ digits in a fixed base $b\geq 2$. Then, for an integer $n\in\cB_L$ with base-$b$ expansion
\begin{equation*}
    n=\sum_{0\leq i<L}\varepsilon_i(n)b^i,\qquad \varepsilon_i(n)\in\{0,\ldots,b-1\},\ \varepsilon_{L-1}(n)\neq 0
\end{equation*}
we write
\begin{equation*}
    \rev{n}=\sum_{0\leq i<L}\varepsilon_i(n)b^{L-1-i}
\end{equation*}
for the \emph{digital reverse} of $n$. As mentioned earlier, it is often easier to restrict to studying reversed numbers in
\begin{align}\label{Bstardefs}
    \mathcal{B}_L^*:=\{n\in\mathcal{B}_L:(n,b^3-b)=1\}
\end{align}
where we have used the standard notation $(a,b):=\gcd(a,b)$.

Now, in order to prove Theorems~\ref{tern1thm:intro}--\ref{sqfreethm:intro}, we require bounds for the number of reversed primes in arithmetic progressions. In the literature, this has been studied through the counting functions
\begin{equation*}
    \rev{\pi}_L(a,q):=\sum_{\substack{p\in\cB_L\\\rev{p}\equiv a\thinspace(\text{mod}\ q)}}1,\qquad \rev{\vartheta}_L(a,q):=\sum_{\substack{p\in\cB_L\\\rev{p}\equiv a\thinspace(\text{mod}\ q)}}\log p
\end{equation*}
or very similar variations thereof. In~\cite{bhowmik2024telhcirid}, an asymptotic for $\rev{\pi}_L(a,q)$, which the authors called the ``Zsiflaw--Legeis" theorem, was obtained for all bases $b\geq 31699$, and this was improved to $b\geq 2$ in both~\cite{dartyge2025prime} and~\cite{bhowmik2025zsiflaw}. These asymptotics are somewhat complicated, owing to $\rev{\pi}_L(a,q)$ and $\rev{\vartheta}_L(a,q)$ counting all reversed primes $\rev{p}$ rather than just those with $(\rev{p},b^3-b)=1$. We thus first take the opportunity to give a variant of the Zsiflaw--Legeis theorem for reversed primes coprime to $b^3-b$. Our result will be stated as an asymptotic for
\begin{equation}\label{thetastardef}
    \rev{\vartheta}^*_L(a,q):=\sum_{\substack{\rev{p}\in\cB_L^*\\\rev{p}\equiv a\thinspace(\text{mod}\ q)}}\log p.
\end{equation}
Notably, here and in what follows we state all of our results with logarithmic weights, since these are neater to work with. However, unweighted variants can readily be obtained via partial summation.
\begin{theorem}\label{thm:modifiedZsiflaw}
    Let $b,a,q\in\Z$ with $b\geq 2$ and $q\geq 1$. For sufficiently large $L\in\mathbb{N}$,
    \begin{equation}\label{Dartygenew}
        \rev{\vartheta}^*_L(a,q)=\frac{\varphi(b)}{b}\frac{(q,b^3-b)}{\varphi((q,b^3-b))}\thinspace \frac{\rho_b(a,q)}{q}b^{L}+O_b\left(b^L\exp\left(-c\sqrt{L}\right)\right)
    \end{equation} 
    provided
    \begin{equation*}
        q\leq\exp(c\sqrt{L}),
    \end{equation*}
    where $c>0$ is an effectively computable constant depending on $b$, $\varphi(\cdot)$ is the Euler totient function, and $\rho_b(a,q)$ is given by
    \begin{equation}\label{rhobdef}
        \rho_b(a,q):=
        \begin{cases}
            1,&\text{if $(a,q,b^3-b)=1$,}\\
            0,&\text{if $(a,q,b^3-b)>1$}.
        \end{cases}
    \end{equation}
\end{theorem}
Theorem~\ref{thm:modifiedZsiflaw} is a generalisation of~\cite[Lemma~2.1]{chourasiya2025power}, which only considered $a=0$ and was used to prove an asymptotic for the number of square-free reversed primes. 

Compared to the usual Siegel--Walfisz theorem for primes, Theorem~\ref{thm:modifiedZsiflaw} is effective, has a stronger error term, and holds for a wider range of $q$. The main drawback though, is that the definition of $\rev{\vartheta}_L^*(a,q)$ requires one to count all reversed primes of length $L$. This causes problems in applications where one often requires finer detail about the distribution of reversed primes. As mentioned in~\cite[p.~2]{bhowmik2025zsiflaw}, one needs to fix the digit length so as to avoid the ``truly wild" behaviour that would occur if one instead counted over $p\leq x$. However here, we explore a different perspective. Namely, we find that if one counts over $\rev{p}\leq x$, then a very neat distribution is obtained of a similar form to~\eqref{Dartygenew}. In particular, we consider
\begin{equation}\label{thetastardef2}
    \rev{\vartheta}^*(x;a,q):=\sum_{\substack{\rev{p}\leq x\\\rev{p}\equiv a\thinspace(\text{mod}\ q)\\ (\rev{p},b^3-b)=1}}\log p
\end{equation}
and define
\begin{equation*}
    \mathfrak{B}(x):=\{n\leq x:(\rev{n},b)=1\}
\end{equation*}
for the set of integers $n\leq x$ whose leading digit is coprime to the choice of base $b$. Notably,
\begin{equation*}
    \#\fB(b^L)-\#\fB(b^{L-1})=\frac{\varphi(b)}{b}b^L,   
\end{equation*}
which is a factor in~\eqref{Dartygenew}. Our new asymptotic result, expressed in terms of $\#\fB(x)$, is then as follows.
\begin{theorem}[Refined Zsiflaw--Legeis theorem]\label{thm: smooth zsiflaw legeis}
    Let $b,a,q\in\Z$ with $b\geq 2$ and $q\geq 1$. Then, for all $A>0$
    \begin{equation}\label{smoothasym}
        \rev{\vartheta}^*(x;a,q)=\frac{(q,b^3-b)}{\varphi((q,b^3-b))}\frac{\rho_b(a,q)}{q} \#\fB(x) + O_{b,A}\bigg(\frac{x}{(\log  x)^A}\bigg)
    \end{equation}
    provided
    \begin{equation}\label{qrangeeq}
        q\leq\exp(c\sqrt{L}),
    \end{equation}
    where $c>0$ is an ineffective constant depending on $b$, and $\rho_b(a,q)$ is as in~\eqref{rhobdef}.
\end{theorem}
Theorem~\ref{thm: smooth zsiflaw legeis} turns out to be very useful in additive applications. In the context of Theorems~\ref{tern1thm:intro}--\ref{sqfreethm:intro}, the asymptotic~\eqref{smoothasym} allows us to accurately count the number of reversed primes $\rev{p}\leq x$ with $x\leq N$. This is very important in dealing with the major arcs in the standard form of the circle method.

Compared to Theorem~\ref{thm:modifiedZsiflaw}, our proof of Theorem~\ref{thm: smooth zsiflaw legeis} actually uses the classical Siegel--Walfisz theorem for primes, leading to a weaker and ineffective error term in~\eqref{smoothasym}. The range of $q$ in~\eqref{qrangeeq} however, is the same as in Theorem~\ref{thm:modifiedZsiflaw}.

We conclude this subsection by remarking that the error term in~\eqref{smoothasym} is likely much stronger in reality. We conjecture that in fact an $O(x^{1/2+\varepsilon})$ error term should be possible, in-line with the usual (generalised) Riemann hypothesis for primes in arithmetic progressions\footnote{For a proof of this Riemann hypothesis for reversed primes, or ``Nnameir" hypothesis, the authors of this paper are willing to offer $\$\rev{1000000}$ (AUD).}. 

\subsection{Asymptotics for Theorems~\ref{tern1thm:intro},~\ref{tern2thm:intro} and~\ref{sqfreethm:intro} and Conjecture~\ref{hcabcon}}\label{Appsect}
We now state, and later prove, asymptotic versions of Theorems~\ref{tern1thm:intro},~\ref{tern2thm:intro} and \ref{sqfreethm:intro}, and a conjectural asymptotic for Hcabdlog's conjecture (Conjecture~\ref{hcabcon}). Each of these results are applications of our refined Zsiflaw--Legeis theorem (Theorem~\ref{thm: smooth zsiflaw legeis}).

We begin with asymptotics for our ternary results, Theorems~\ref{tern1thm:intro} and~\ref{tern2thm:intro}.
\begin{theorem}\label{ternaryhcabthm1}
    Let $b\geq 2$ be fixed. Then every sufficiently large odd integer $N$ can be expressed as the sum of a prime and two reversed primes. In particular, if
    \begin{equation*}
        \mathcal{R}_{1,2}(N):=\sum_{\substack{p_1,\rev{p_2},\rev{p_3}\leq N\\ N=p_1+\rev{p_2}+\rev{p_3}\\ (\rev{p_2}\rev{p_3},b^3-b)=1}}\log p_1\log p_2 \log p_3,
    \end{equation*}
    then, for any $A>0$,
    \begin{equation}\label{r12asym}
        \mathcal{R}_{1,2}(N)=\mathfrak{S}_3(N)\cS_{1,2}(N)+O_{b,A}\left(\frac{N^{2}}{(\log N)^A}\right),
    \end{equation}
    where
    \begin{equation}\label{sing1eq}
        \mathfrak{S}_3(N):=\prod_{\substack{p\mid b^3-b\\ p\mid N}}\left(1-\frac{1}{(p-1)^2}\right)\prod_{\substack{p\mid b^3-b\\ p\nmid N}}\left(1+\frac{1}{(p-1)^3}\right).
    \end{equation}
    and
    \begin{equation}\label{S12def}
       \cS_{1,2}(N):=\sum_{\substack{n_1,n_2,n_3\leq N\\n_1+n_2+n_3=N\\(\rev{n_2}\rev{n_3},b)=1}}1. 
    \end{equation}
\end{theorem}

\begin{theorem}\label{ternaryhcabthm2}
    Let $b\geq 2$ be fixed. Then every sufficiently large odd integer $N$ can be expressed as the sum of two primes and a reversed prime. In particular, if
    \begin{equation*}
        \mathcal{R}_{2,1}(N):=\sum_{\substack{p_1,p_2,\rev{p_3}\leq N\\ N=p_1+p_2+\rev{p_3}\\ (\rev{p_3},b^3-b)=1}}\log p_1\log p_2 \log p_3,
    \end{equation*}
    then, for any $A>0$,
    \begin{equation}\label{r21asym}
        \mathcal{R}_{2,1}(N)=\mathfrak{S}_3(N)\cS_{2,1}(N)+O_{b,A}\left(\frac{N^2}{(\log N)^A}\right),
    \end{equation}
    where $\mathfrak{S}_3(N)$ is as in~\eqref{sing1eq}, and
    \begin{equation}\label{S21def}
       \cS_{2,1}(N):=\sum_{\substack{n_1,n_2,n_3\leq N\\n_1+n_2+n_3=N\\(\rev{n_3},b)=1}}1.  
    \end{equation}
\end{theorem}
In Theorems~\ref{ternaryhcabthm1} and~\ref{ternaryhcabthm2}, it is useful to note that (see Lemma~\ref{S12S21lem})
\begin{equation}\label{Sasym}
    \cS_{1,2}(N),\:\cS_{2,1}(N)\asymp_b N^2
\end{equation}
and for odd $N$,
\begin{equation*}
    \fS_3(N)>\prod_{p>2}\left(1-\frac{1}{(p-1)^2}\right)=C_2>0,
\end{equation*}
where $C_2=0.66016\ldots$ is the twin prime constant. This ensures that~\eqref{r12asym} and~\eqref{r21asym} are asymptotic expresions.

Theorems~\ref{ternaryhcabthm1} and~\ref{ternaryhcabthm2} are proven using Theorem~\ref{thm: smooth zsiflaw legeis} and the Hardy--Littlewood circle method. Namely, by modifying the classical argument for Vinogradov's three prime theorem. In general, the asymptotics~\eqref{r12asym} and ~\ref{r21asym} readily generalise to sums of $k_1\geq 1$ primes and $k_2\geq 0$ reversed primes, provided $k_1+k_2\geq 3$. By only considering the major arcs in the circle method, we also get the following asymptotic form of Conjecture~\ref{hcabcon}. 
\begin{conjecture}[Asymptotic Hcabdlog's conjecture]\label{asymhcabcon}
    Let $b\geq 2$ be fixed. Then every sufficiently large even integer $N$ can be expressed as the sum of a prime and a reversed prime. In particular, if
    \begin{equation}\label{R11def}
        \mathcal{R}_{1,1}(N):=\sum_{\substack{p_1,\rev{p_2}\leq N\\ N=p_1+\rev{p_2}\\ (\rev{p_2},b^3-b)=1}}\log p_1\log p_2,
    \end{equation}
    then
    \begin{equation*}
        \mathcal{R}_{1,1}(N)\sim\mathfrak{S}_2(N)\#\fB(N),
    \end{equation*}
    where
    \begin{equation}\label{sing2eq}
        \fS_2(N):=\prod_{\substack{p\mid b^3-b\\ p\mid N}}\left(1+\frac{1}{p-1}\right)\prod_{\substack{p\mid b^3-b\\ p\nmid N}}\left(1-\frac{1}{(p-1)^2}\right).
    \end{equation}
\end{conjecture}

Finally, we state the asymptotic version of Theorem~\ref{sqfreethm:intro} on the sum of a reversed prime and a square-free number. In what follows, $\mu^2(\cdot)$ is the usual square-free indicator function.
\begin{theorem}\label{sqfreethm}
    Fix $b\geq 2$. Then all sufficiently large integers $N$ can be expressed as the sum of a reversed prime and a square-free number. More precisely, if
    \begin{equation*}
        \mathcal{R}_{\square}(N):=\sum_{\substack{1\leq \rev{p}\leq N\\ \mu^2(N-\rev{p})=1\\(\rev{p},b^3-b)=1}}\log p,
    \end{equation*}
    then, for any $A>0$,
    \begin{equation}\label{squarerep}
        \cR_{\square}(N)=\frac{\#\mathfrak{B}(N)}{\zeta(2)}\fS_{\square}(N)+O_{b,A}\left(\frac{N}{(\log N)^A}\right),
    \end{equation}
    where $\zeta(\cdot)$ is the Riemann-zeta function, and
    \begin{equation*}
        \fS_{\square}(N)=\prod_{\substack{p\mid b^3-b}}\left(1+\frac{1}{p^2-1}\right)\prod_{\substack{p\mid b^3-b\\ p\nmid N}}\left(1-\frac{1}{p^2-p}\right).
    \end{equation*}
\end{theorem}
Notably, $\fS_{\square}(N)\neq 0$ for both odd and even $N$ and Theorem~\ref{sqfreethm} approximates Hcabdlog's conjecture since primes are square-free. Moreover, the proof of Theorem~\ref{sqfreethm} does not use the circle method and is simpler than that of Theorems~\ref{ternaryhcabthm1},~\ref{ternaryhcabthm2} and~\ref{almostthm:intro}. In fact, to obtain the weaker result $\cR_{\square}(N)>0$ one may use Theorem~\ref{thm:modifiedZsiflaw} rather than our refined Zsiflaw--Legeis theorem (Theorem~\ref{thm: smooth zsiflaw legeis}). Such an approach may be desirable if one is interested in an effective result.

\subsection{Outline of paper}\label{outlinesect}
An outline of the rest of the paper is as follows. In Section~\ref{precursorsect}, we expand on the techniques in~\cite{bhowmik2024telhcirid} and~\cite{bhowmik2025zsiflaw} to prove a precursor result (Proposition~\ref{prop: partitioned Zsiflaw Legeis}) which describes the distribution of reversed primes in an interval $(rb^{i},(r+1)b^i)$ for certain ranges of $r$ and $i$. We then use this result in Section~\ref{smoothsect} to prove our variants of the Zsiflaw--Legeis theorem (Theorems~\ref{thm:modifiedZsiflaw} and~\ref{thm: smooth zsiflaw legeis}). Next, in Section~\ref{ternsect} we prove Theorems~\ref{ternaryhcabthm1},~\ref{ternaryhcabthm2} and~\ref{almostthm:intro} using the circle method in conjunction with Theorem~\ref{thm: smooth zsiflaw legeis}. This is followed by the proof of Theorem~\ref{sqfreethm} in Section~\ref{sqfreesect}. We conclude in Section~\ref{sumsect} with further discussion and conjectures in relation to representing integers purely as the sum of reversed primes, as opposed to a mix of primes and reversed primes.

Throughout we employ Vinogradov's notation $f\ll g$ to mean $f=O(g)$, and often use the shorthand $a\equiv b\thinspace (q)$ to mean $a\equiv b\thinspace(\text{mod}\ q)$. We also write $\mu(d)$ for the M\"obius function, $||x||=\min_{n\in\mathbb{Z}}|x-n|$ for the distance between $x$ and the nearest integer, and $e(x):=\exp(2\pi ix)$ for complex exponentials. Any use of the variable $p$ or any subscript thereof (e.g.~$p_1,p_2,p_3$) is assumed to denote a prime.

\section{A precursor to Theorems~\ref{thm:modifiedZsiflaw} and~\ref{thm: smooth zsiflaw legeis}}\label{precursorsect}
Fix a base $b\geq 2$. In this section we study the distribution of the function
\begin{equation}\label{thetaetadef}
    \rev{\vartheta}^*_L(\eta,r,a,q):=\sum_{\substack{rb^{L-\eta}\leq \rev{p}<(r+1)b^{L-\eta}\\\rev{p}\equiv a(q)\\ (\rev{p},b^3-b)=1}}\log p
\end{equation}
for $L,\eta\in\mathbb{N}$ with $\eta\leq L$ and $b^{\eta-1}\leq r < b^\eta$. In particular~$\rev{\vartheta}^*_L(\eta,r,a,q)$ counts $L$-digit reversed primes in arithmetic progressions whose first $\eta$ digits are equal to $r$. Our main goal will be to obtain the following asymptotic, from which we deduce Theorems~\ref{thm:modifiedZsiflaw} and~\ref{thm: smooth zsiflaw legeis} in the following section.

\begin{proposition}\label{prop: partitioned Zsiflaw Legeis}
    Let $b,a,q\in\Z$ with $b\geq 2$ and $q\geq 1$. Moreover, let $L\in\N$ and $b^\eta\leq L^\ell$ for any fixed $\ell>0$. Then, for any $b^{\eta-1}\leq r<b^\eta$ we have
    \begin{equation}\label{revetaasym}
        \rev{\vartheta}_L^*(\eta,r,a,q)=\kappa_b(\rev{r})\rho_b(a,q)\frac{(q,b^3-b)}{\varphi((q,b^3-b))}\frac{b^{L-\eta}}{q} +O_{b,\ell}\bigg(b^L \exp(-c\sqrt{L})\bigg)
    \end{equation}
    provided
    \begin{equation*}
         q\leq\exp(c\sqrt{L}),
    \end{equation*}
    where $c>0$ is an ineffective constant depending on $b$, $\rho_b(a,q)$ is as in~\eqref{rhobdef}, and
    \begin{equation*}
        \kappa_b(\rev{r}):=\begin{cases}
            1 & \text{if } (\rev{r},b)=1,\\
            0 & \text{otherwise.}
        \end{cases}
    \end{equation*}
\end{proposition}
\begin{remark}
    The ineffectivity of the constant $c$ is due to the application of the Siegel--Walfisz theorem for primes in arithmetic progressions modulo $b^{\eta}\leq L^{\ell}$.
    However, if $\eta$ is a constant then $c$ can be taken to be effective by instead applying an effective form of the prime number theorem for arithmetic progressions for \emph{bounded} moduli (e.g.~\cite[Corollary~11.7]{montgomery2006multiplicative}).
\end{remark}

In order to prove Proposition~\ref{prop: partitioned Zsiflaw Legeis}, we make the key observation that counting $L$-digit reversed primes which have their first $\eta$ digits equal to $r$ is equivalent to counting $L$-digit primes $p$ with $p\equiv\rev{r}$ (mod $b^\eta$). That is,

\begin{align}    
        \rev{\vartheta}^*_L(\eta,r,a,q)=\sum_{\substack{\rev{p}\in\mathcal{B}_L^*\\\rev{p}\equiv a\thinspace(q)\\ p\equiv \rev{r}(b^\eta)}}\log p \label{newrevstardef}.
\end{align}

To account for the conditions~$(\rev{p},b^3-b)=1$ and~$p\equiv\rev{r}\thinspace(b^{\eta})$ in~\eqref{newrevstardef}, we essentially have to rework the main results of the papers~\cite{bhowmik2024telhcirid,bhowmik2025zsiflaw}. This is a rather technical task, which we have made an effort to do as efficiently as possible. Notably, we are able to prove the key error term estimate (see Lemma~\ref{newerrlem}) by utilising Bhowmik and Suzuki's ``weakly digital" framework in~\cite{bhowmik2025zsiflaw}. This will be done in Subsection~\ref{interludesect}, followed by the completion of the proof of Proposition~\ref{prop: partitioned Zsiflaw Legeis} in Subsection~\ref{everythingsect}.

First though, we prove a simple lemma, which we apply several times throughout this section. This lemma and its proof encapsulates the reason why $b^2-1$ is such an important modulus when working with digital reverses. 
\begin{lemma}\label{b21lem}
    For any $b\geq 2$ and $n>0$, we have
    \begin{equation*}
        (n,b^2-1)>1\quad\text{if and only if}\quad (\rev{n},b^2-1)>1. 
    \end{equation*}
    In particular, there are only finitely many primes $p$ with $(\rev{p},b^2-1)>1$, each of which satisfy $p\mid b^2-1$.
\end{lemma}
\begin{proof}
    Suppose $n$ has $L$ digits so that the base-$b$ expansions of $n$ and $\rev{n}$ are
    \begin{equation*}
        n=\sum_{0\leq i<L}\varepsilon_i(n)b^{i},\quad \rev{n}=\sum_{0\leq i<L}\varepsilon_i(n) b^{L-1-i}.
    \end{equation*}
    Since $b^2\equiv 1\thinspace\text{(mod $b^2-1$)}$, we have $b\equiv b^{-1}\thinspace\text{(mod $b^2-1$)}$ and thus
    \begin{align*}
        \rev{n}&\equiv\sum_{0\leq i<L}\varepsilon_i(n) b^{L-1+i}\pmod{b^2-1}\\
        &\equiv b^{L-1} n\pmod{b^2-1},
    \end{align*}
    from which the desired result follows.
\end{proof}

\subsection{Exponential sum bounds and weakly digital functions}\label{interludesect}
In both~\cite{bhowmik2025zsiflaw} and~\cite{dartyge2025prime}, the key intermediary result required to prove their versions of the Zsiflaw--Legeis theorem is the following exponential sum bound.
\begin{lemma}[{see \cite[Theorem~1.7]{dartyge2025prime}}]\label{dartlem}
    For any $b\geq 2$, there exists $c=c(b)>0$ such that for any integer $L\geq 2$ and any $(h,q)\in\Z^2$ with $2\leq q\leq \exp(c\sqrt{L})$ and $q\nmid (b^2-1)b^Lh$, we have
    \begin{equation}\label{errorlemeq}
        \sup_{t\in[b^{L-1},b^L]}\left|\sum_{b^{L-1}\leq n<t}\Lambda(n)e\left(\frac{h\rev{n}}{q}\right)\right|\ll_b b^{L}\exp\left(-c\sqrt{L}\right),
    \end{equation}
    where $\Lambda(n)$ is the von Mangoldt function:
    \begin{equation*}
        \Lambda(n)=
        \begin{cases}
            \log p,&\text{if $n=p^k$ for some prime p},\\
            0,&\text{otherwise.}
        \end{cases}
    \end{equation*}
\end{lemma}

To account for the congruence $p\equiv \rev{r}$ $(b^{\eta})$ in~\eqref{newrevstardef} we require that the following generalisation of Proposition~\ref{dartlem} holds.

\begin{lemma}\label{newerrlem}
    For any $b\geq 2$, there exists $c=c(b)>0$ such that for any integer $L\geq 2$ and any $(h,q)\in\Z^2$ with $2\leq q\leq \exp(c\sqrt{L})$ and $q\nmid (b^2-1)b^Lh$, we have
    \begin{equation}
        \sup_{\substack{t\in[b^{L-1},b^L]\\(d,f)\in\mathbb{Z}^2}}\left|\sum_{\substack{b^{L-1}\leq n<t\\ n\equiv f\thinspace(d)}}\Lambda(n)e\left(\frac{h\rev{n}}{q}\right)\right|\ll_b b^{L}\exp\left(-c\sqrt{L}\right).
    \end{equation}
\end{lemma}

In particular, we can recover the same bound as~\eqref{errorlemeq} after any congruence condition is placed on $n$. Fortunately, the proof of Lemma~\ref{newerrlem} is simple if one utilises the general framework in~\cite{bhowmik2025zsiflaw} on \emph{weakly digital} functions, defined as follows.

\begin{definition}{{\cite[Definition~1]{bhowmik2025zsiflaw}}}
    For $b\geq 2$, let $\cA_b$ be the set of sequences of maps
    \begin{equation*}
        \boldsymbol{\alpha}=(\alpha_i:\{0,\ldots ,b-1\}\to\R)_{i=0}^\infty,
    \end{equation*}
    and for any $n\geq 0$ write 
    \begin{equation*}
        n=\sum_{i\geq 0}\varepsilon_i(n)b^i
    \end{equation*}
    for the base-$b$ expansion of $n$. For $\boldsymbol{\alpha}\in\cA_b$ and $\lambda\in\Z_{\geq 0}$, let $f_{\lambda}=f_{\lambda,\boldsymbol{\alpha}}:\Z_{\geq 0}\to\R$ be defined by
    \begin{equation*}
        f_{\lambda,\boldsymbol{\alpha}}=\sum_{0\leq i<\lambda}\alpha_i(\varepsilon_i(n)).
    \end{equation*}
    The functions $f_{\lambda,\boldsymbol{\alpha}}$ are called \emph{weakly digital functions generated by the seed $\boldsymbol{\alpha}$.}
\end{definition}

With this definition, we can state Bhowmik and Suzuki's more general version of  Lemma~\ref{dartlem}. Here, one also requires the notation, for any $\boldsymbol{\alpha}\in\cA_b$,
\begin{equation*}
    \sigma_{\lambda}(\boldsymbol{\alpha}):=\sum_{0\leq i<\lambda}\gamma_i(\boldsymbol{\alpha}),
\end{equation*}
where
\begin{equation}\label{gammaieq}
    \gamma_i(\boldsymbol{\alpha}):=\frac{2\log 2}{2(b-1)b^4(\log b)^2}\sum_{0\leq m<n<b}||(b\alpha_i(m)-\alpha_{i+1}(m))-(b\alpha_i(n)-\alpha_{i+1}(n))||^2.
\end{equation}

\begin{lemma}[{\cite[Theorem 3]{bhowmik2025zsiflaw}}]\label{kappalem}
    For $b\geq 2$, $\boldsymbol{\alpha}\in\cA_b$, $L\in\mathbb{N}$, $2\leq x\leq b^L$, we have
    \begin{equation}\label{kappaeq}
        \sum_{n\leq x}\Lambda(n)e(f_L(n))\ll_b xb^{-\kappa}(\log x)^4,
    \end{equation}
    where $f_L(n)$ is the weakly digital function with seed $\boldsymbol{\alpha}$,
    \begin{equation*}
        \kappa:=\frac{1}{10}\sigma_{\xi}(\boldsymbol{\alpha})\quad\text{with}\quad\xi:=\left[\frac{1}{4}\frac{\log x}{\log b}\right],
    \end{equation*}
    and the implied constant in~\eqref{kappaeq} depends only on $b$.
\end{lemma}
Using Lemma~\ref{kappalem}, we now proceed with the proof of Lemma~\ref{newerrlem} by encoding the congruence condition $n\equiv f\thinspace(d)$ into a weakly digital function.
\begin{proof}[Proof of Lemma~\ref{newerrlem}]
    As in~\cite[\S 8]{bhowmik2025zsiflaw}, we write
    \begin{equation*}
        \trev_L(n):=\sum_{0\leq i<L}\varepsilon_i(n)b^{L-i-1}
    \end{equation*}
    so that $\trev_L(n)$ agrees with $\rev{n}$ on the set $[b^{L-1},b^L)$. Now, for any $2\leq x\leq b^L$ and $(d,f)\in\Z^2$, consider the sum
    \begin{align}
        \sum_{\substack{n\leq x\\n\equiv f\thinspace(d)}}\Lambda(n)e\left(\frac{h\:\trev_L(n)}{q}\right)&=\frac{1}{d}\sum_{k=1}^de\left(-\frac{kf}{d}\right)\sum_{n\leq x}\Lambda(n)e\left(\frac{h\:\trev_L(n)}{q}\right)e\left(\frac{kn}{d}\right)\notag\\
        &\ll\sup_{k\in\Z}\left|\sum_{n\leq x}\Lambda(n)e\left(\frac{h\:\trev_L(n)}{q}+\frac{kn}{d}\right)\right|\label{supkeq}.
    \end{align}
    We now set
    \begin{equation*}
        \alpha_{i}(n)=\frac{h}{q}nb^{L-i-1}+\frac{k}{d}n b^i,\qquad\boldsymbol{\alpha}=(\alpha_{i})_{i=1}^{L-1}
    \end{equation*}
    and
    \begin{equation*}
        f_{L}(n)=\sum_{0\leq i<L}\alpha_{i}(\varepsilon_i(n)),
    \end{equation*}
    to be the corresponding weakly digital function. In particular, \eqref{supkeq} is equal to
    \begin{equation}\label{supkeq2}
        \sup_{k\in\Z}\left|\sum_{n\leq x}\Lambda(n)e\left(f_L(n)\right))\right|.
    \end{equation}
    This expression can be directly bounded using Lemma~\ref{kappalem}. All that we need to compute is
    \begin{equation*}
        \kappa:=\frac{1}{10}\sigma_{\xi}(\boldsymbol{\alpha})=\frac{1}{10}\sum_{0\leq i<\lambda}\gamma_i(\boldsymbol{\alpha}),
    \end{equation*}
    with $\gamma_i$ as defined in \eqref{gammaieq}. Considering the $m=0$ and $n=1$ terms in~\eqref{gammaieq} gives
    \begin{align*}
        \gamma_i(\boldsymbol{\alpha})&\gg_b||b\alpha_i(0)-\alpha_{i+1}(0)-b\alpha_i(1)+\alpha_{i+1}(1)||^2\\
        &=\left|\left|b^{L-i-2}(b^2-1)\frac{h}{q}\right|\right|\\&\gg_b\left|\left|b^{L-i-1}(b^2-1)\frac{h}{q}\right|\right|.
    \end{align*}
    With this lower bound for $\gamma_i(\boldsymbol{\alpha)}$, one gets a lower bound for $\kappa$ by~\cite[Lemma~26]{bhowmik2025zsiflaw} (or the proof thereof), whereby it is shown that
    \begin{equation*}
        \kappa\gg\sigma_{\xi}(\boldsymbol{\alpha})\gg_b\frac{\xi}{\log\frac{1}{\sigma}}+O_b(1),
    \end{equation*}
    with
    \begin{equation*}
        \sigma:=\min_{0\leq i\leq L}\left|\left|b^i(b^2-1)\frac{h}{q}\right|\right|.
    \end{equation*}
    Since we are assuming $q\nmid (b^2-1)b^Lh$, it follows that $\sigma\gg_b 1/q$ and thus
    \begin{equation*}
        \kappa\gg_b\frac{\xi}{\log q}\gg_b\frac{\log x}{\log b}.
    \end{equation*}
    Therefore, applying Lemma~\ref{kappalem} to the expression~\eqref{supkeq2} gives
    \begin{equation*}
        \sup_{k\in\Z}\left|\sum_{n\leq x}\Lambda(n)e\left(f_L(n)\right))\right|\ll_b x(\log x)^4\exp\left(-c'\frac{\log x}{\log q}\right)
    \end{equation*}
    for some constant $c'$ depending on $b$. Taking $c=\sqrt{c'}\log b-\delta$ for any small $\delta>0$, the above bound reduces to
    \begin{equation*}
        \sup_{k\in\Z}\left|\sum_{n\leq x}\Lambda(n)e\left(f_L(n)\right))\right|\ll_b x\exp\left(-c\sqrt{L}\right)
    \end{equation*}
    for all $q\leq\exp(c\sqrt{L})$. The proposition then follows from an application of the triangle inequality. That is, for any $t\in[b^{L-1},b^L]$ and $q\leq\exp(c\sqrt{L})$,
    \begin{align*}
        \left|\sum_{\substack{b^{L-1}\leq n<t\\ n\equiv f\thinspace(d)}}\Lambda(n)e\left(\frac{h\rev{n}}{q}\right)\right|&\leq \left|\sum_{\substack{n<t\\ n\equiv f\thinspace(d)}}\Lambda(n)e\left(\frac{h\:\trev_L(n)}{q}\right)\right|+\left|\sum_{\substack{n<b^{L-1}\\ n\equiv f\thinspace(d)}}\Lambda(n)e\left(\frac{h\:\trev_L(n)}{q}\right)\right|\\
        &\ll_b b^L\exp(-c\sqrt{L}).\qedhere
    \end{align*}
\end{proof}
\subsection{The proof of Proposition~\ref{prop: partitioned Zsiflaw Legeis}}\label{everythingsect}
We now proceed with the proof of Proposition~\ref{prop: partitioned Zsiflaw Legeis}. Essentially, we extract the main term in~\eqref{revetaasym} by using similar arguments as in~\cite[\S8 and \S11]{bhowmik2024telhcirid}. However, extra care is taken to account for the coprimality condition $(\rev{p},b^3-b)$, and the Siegel--Walfisz theorem is used to handle the congruence condition $p\equiv\rev{r}\thinspace(q)$ appearing in~\eqref{newrevstardef}. The error term is then bounded using Lemma~\ref{newerrlem}.
\begin{proof}[{Proof of Proposition~\ref{prop: partitioned Zsiflaw Legeis}}]
    Throughout we take $q\geq 2$ since if $q=1$ then~\eqref{revetaasym} follows from directly applying the Siegel--Walfisz theorem to~\eqref{newrevstardef}. Now, consider the case when $(\rev{r},b)>1$. Under this condition, we have
    \begin{equation*}
        \rev{\vartheta}^*_L(\eta,r,a,q)\leq \log b^L\ll_bL
    \end{equation*}
    since any prime $p$ counted by \eqref{newrevstardef} satisfies $p\equiv\rev{r}\thinspace (b^L)$ and thus $(p,b)>1$.

    The other exceptional case is when $(a,q,b^3-b)>1$. In this scenario, one has
    \begin{equation*}
        \rev{\vartheta}^*_L(\eta,r,a,q)=0
    \end{equation*}
    since if $(a,q,b^3-b)>1$ and ${\rev{p}\equiv a\thinspace (q)}$ then $(\rev{p},b^3-b)>1$. That is, there are no reversed primes counted by \eqref{newrevstardef}.

    By the above considerations, we may from now on assume that $(\rev{r},b)=1$ and $(a,q,b^3-b)=1$. That is, it suffices to prove the proposition in the case $\kappa(\rev{r})=1$ and $\rho_b(a,q)=1$. Arguing as in \cite[\S 8]{bhowmik2024telhcirid}, we use exponential sums to estimate $\rev{\vartheta}^*_L(\eta,r,a,q)$. In particular, we write
    \begin{align}\label{thetabigeq}
        \rev{\vartheta}^*_L(\eta,r,a,q)& = \frac{1}{q} \sum_{0 \leq h < q} e\left(-\frac{ha}{q}\right)\sum_{\substack{\rev{p}\in \mathcal{B}^{*}_{L}\\p\equiv\rev{r}(b^\eta)}} e \left(\frac{h \overleftarrow{p} }{q} \right)\log p\notag\\
        &=\frac{1}{q} \sum_{\substack{0 \leq h < q\\ q\mid(b^2-1)b^L h}} e\left(-\frac{ha}{q}\right)\sum_{\substack{\rev{p} \in \mathcal{B}^{*}_{L}\\p\equiv\rev{r}(b^\eta)}} e \left(\frac{h \overleftarrow{p} }{q} \right)\log p\notag\\
        &\qquad\qquad\qquad\qquad\qquad+\frac{1}{q} \sum_{\substack{0 \leq h < q\\ q\nmid(b^2-1)b^L h}} e\left(-\frac{ha}{q}\right)\sum_{\substack{\rev{p} \in \mathcal{B}^{*}_{L}\\p\equiv\rev{r}(b^\eta)}} e \left(\frac{h \overleftarrow{p} }{q} \right)\log p .
    \end{align}
    We now show that the first term in \eqref{thetabigeq} gives our desired main term, and that the second term in \eqref{thetabigeq} gives a sufficiently small error term. We begin with an analysis of the error term, given by
    \begin{align}\label{Ejtaqdef}
        E(\eta,r,a,q)&:=\frac{1}{q} \sum_{\substack{0 \leq h < q\\ q\nmid(b^2-1)b^L h}} e\left(-\frac{ha}{q}\right)\sum_{\substack{\rev{p} \in \mathcal{B}^{*}_{L}\\p\equiv\rev{r}(b^\eta)}} e \left(\frac{h \overleftarrow{p} }{q} \right)\log p\notag\\
        &\ll\sup_{\substack{(h,q)\in\mathbb{Z}^2\\q\nmid (b^2-1)b^Lh}}\left|\sum_{\substack{\rev{p} \in \mathcal{B}^{*}_{L}\\p\equiv\rev{r}(b^\eta)}} e \left(\frac{h \overleftarrow{p} }{q} \right)\log p\right|.
    \end{align}
    Recall that the condition $\rev{p}\in\mathcal{B}_L^*$ means that $(\rev{p},b^3-b)=1$, or equivalently, $(\rev{p},b^2-1)=1$ and $(\rev{p},b)=1$. By Lemma~\ref{b21lem} there are only finitely many primes $p$ with $(\rev{p},b^2-1)>1$. Therefore,
    \begin{equation}\label{Ejtaqbound1}
        E(\eta,r,a,q)\ll\sup_{\substack{(h,q)\in\mathbb{Z}^2\\q\nmid (b^2-1)b^Lh}}\left|\sum_{\substack{p\in\mathcal{B}_L\\(\rev{p},b)=1\\p\equiv\rev{r}(b^\eta)}} e \left(\frac{h \overleftarrow{p} }{q} \right)\log p\right|+O_b(1),
    \end{equation}
    where the $\rev{p}\in\mathcal{B}_L^*$ condition has been replaced by $p\in\mathcal{B}_L$ and $(\rev{p},b)=1$. Next, we observe that $(\rev{p},b)=1$ is equivalent to the first base-$b$ digit of $p$ being coprime to $b$. In particular, 
    \begin{equation}\label{Ejtaqbound2}
        E(\eta,r,a,q)\ll_b\sup_{\substack{1\leq j<b\\(j,b)=1}}\sup_{\substack{(h,q)\in\mathbb{Z}^2\\q\nmid (b^2-1)b^Lh}}\left|\sum_{\substack{jb^{L-1}\leq p<(j+1)b^{L-1}\\p\equiv\rev{r}(b^\eta)}} e \left(\frac{h \overleftarrow{p} }{q} \right)\log p\right|+O_b(1).
    \end{equation}
    With a view to apply Lemma~\ref{newerrlem} we then note that
    \begin{equation}\label{pkbound}
        \sum_{\substack{p^k\leq t\\ k\geq 2}}\log p\leq\log t\sum_{k\leq\log t/\log 2}t^{1/k}\ll \sqrt{t} (\log t)^2
    \end{equation}
    and thus
    \begin{equation*}
        E(\eta,r,a,q)\ll\sup_{\substack{1\leq j<b\\(j,b)=1}}\sup_{\substack{(h,q)\in\mathbb{Z}^2\\q\nmid (b^2-1)b^Lh}}\left|\sum_{\substack{jb^{L-1}\leq n<(j+1)b^{L-1}\\n\equiv\rev{r}(b^\eta)}} e \left(\frac{h \overleftarrow{n} }{q} \right)\Lambda(n)\right|+O_b(\sqrt{t}(\log t)^2).
    \end{equation*}
    Now, by the triangle inequality and Lemma~\ref{newerrlem},
    \begin{align*}
        &\left|\sum_{\substack{jb^{L-1}\leq n<(j+1)b^{L-1}\\n\equiv\rev{r}(b^\eta)}} e \left(\frac{h \overleftarrow{n} }{q} \right)\Lambda(n)\right|\\
        &\qquad=\left|\sum_{\substack{ b^{L-1}\leq n<(j+1)b^{L-1}\\n\equiv\rev{r}(b^\eta)}} e \left(\frac{h \overleftarrow{n} }{q} \right)\Lambda(n)-\sum_{\substack{ b^{L-1}\leq n<jb^{L-1}\\n\equiv\rev{r}(b^\eta)}} e \left(\frac{h \overleftarrow{n} }{q} \right)\Lambda(n)\right|\\
        &\qquad\leq 2\sup_{t\in[b^{L-1},b^L]}\left|\sum_{\substack{b^{L-1}\leq n<t\\n\equiv\rev{r}(b^\eta)}} e \left(\frac{h \overleftarrow{n} }{q} \right)\Lambda(n)\right|+O(1).\\
        &\qquad\ll_b \exp(b^L\exp(-c'\sqrt{L}))
    \end{align*}
    for some constant $c'>0$ depending on $b$ and $q\leq\exp(c'\sqrt{L})$. In particular, we have
    \begin{equation}\label{errfinal}
        E(\eta,r,a,q)\ll_b \exp(b^L\exp(-c'\sqrt{L}))
    \end{equation}
    as desired.
    
    We now consider the main term in \eqref{thetabigeq}:
    \begin{equation*}
        M(\eta,r,a,q):=\frac{1}{q} \sum_{\substack{0 \leq h < q\\ q\mid(b^2-1)b^L h}} e\left(-\frac{ha}{q}\right)\sum_{\substack{\rev{p} \in \mathcal{B}^{*}_{L}\\p\equiv\rev{r}(b^\eta)}} e \left(\frac{h \overleftarrow{p} }{q} \right)\log p.
    \end{equation*}
    Rewriting the divisibility condition as
    \begin{equation*}
        q\mid (b^2-1)b^L h\quad\Longleftrightarrow\quad\frac{q}{(q,(b^2-1)b^L)}\mid h
    \end{equation*}
    gives
    \begin{align}\label{Mjtaqdef}
        M(\eta,r,a,q)&=\frac{1}{q}\sum_{0\leq h\leq (q,(b^2-1)b^L)}e\left(-\frac{ha}{(q,(b^2-1)b^L)}\right)\sum_{\substack{\rev{p} \in \mathcal{B}^{*}_{L}\\p\equiv\rev{r}(b^\eta)}}e\left(\frac{h\rev{p}}{(q,(b^2-1)b^L)}\right)\log p\notag\\
        &=\frac{(q,(b^2-1)b^L)}{q}\sum_{\substack{\rev{p}\in\mathcal{B}_L^*\\p\equiv\rev{r}(b^\eta)\\\rev{p}\equiv a\ \text{mod $(q,(b^2-1)b^L)$}}}\log p.
    \end{align}
    As with our analysis of the error term $E(\eta,r,a,q)$, we note that by Lemma~\ref{b21lem} there are only $O_b(1)$ primes $p$ with $(\rev{p},b^2-1)>1$. In addition, the condition $(\rev{p},b)=1$ means the first base-$b$ digit of $p$ is coprime to $b$. These observations imply
    \begin{equation}\label{firstmaineq}
        M(\eta,r,a,q)=\frac{(q,(b^2-1)b^L)}{q}\sum_{\substack{0\leq j<b\\(j,b)=1}}\:\sum_{\substack{jb^{L-1}\leq p<(j+1)b^{L-1}\\p\equiv\rev{r}(b^\eta)\\\rev{p}\equiv a\ \text{mod $(q,(b^2-1)b^L)$}}}\log p+O_b(1).
    \end{equation}
    In what follows, we argue similarly to the proof of \cite[Corollary~1.2]{bhowmik2024telhcirid}. The main difference here is that we must incorporate the congruence condition $p\equiv\rev{r}\thinspace(b^{\eta})$ and the sum over $j$ coprime to $b$. To begin with, we note that since $b$ is coprime to $b^2-1$, we may use the Chinese remainder theorem to rewrite the inner sum in~\eqref{firstmaineq} as
    \begin{equation*}
        I_j:=\sum_{\substack{jb^{L-1}\leq p<(j+1)b^{L-1}\\p\equiv\rev{r}(b^\eta)\\\rev{p}\equiv a\ \text{mod $(q,(b^2-1)b^L)$}}}\log p=\sum_{\substack{jb^{L-1}\leq p<(j+1)b^{L-1}\\p\equiv\rev{r}(b^\eta)\\ \rev{p}\equiv a\ \text{mod $(q,b^2-1)$}\\ \rev{p}\equiv a\ \text{mod $(q,b^L)$}}}\log p.
    \end{equation*}
    As in the proof of Lemma~\ref{b21lem}, we have $\rev{p}\equiv b^{L-1}p$ (mod $b^2-1$) so that 
    \begin{equation*}
        \rev{p}\equiv a\ \text{mod $(q,(b^2-1))$}\quad\Longleftrightarrow\quad p\equiv \overline{b}^{L-1}a\ \text{mod $(q,(b^2-1))$}, 
    \end{equation*}
    where $\overline{b}$ is the multiplicative inverse of $b$ mod $(q,b^2-1)$. Thus, we may slightly rewrite $I_j$ as
    \begin{equation}\label{secondmaineq}
        I_j=\sum_{\substack{jb^{L-1}\leq p<(j+1)b^{L-1}\\p\equiv\rev{r}(b^\eta)\\ p\equiv \overline{b}^{L-1}a\ \text{mod $(q,b^2-1)$}\\ \rev{p}\equiv a\ \text{mod $(q,b^L)$}}}\log p.
    \end{equation}
    Next we deal with the $\rev{p}\equiv a$ (mod $(q,b^L)$) condition. To do so, we let $L_0\geq 0$ be the smallest natural number such that
    \begin{equation*}
        (q,b^{L_0})=(q,b^L).
    \end{equation*}
    Then, we may recast \eqref{secondmaineq} as
    \begin{equation}\label{thirdmaineq}
        I_j=\sum_{\substack{1\leq v<b^{L_0}\\ v\equiv j\thinspace (b)\\ v \equiv a\ \text{mod $(q,b^{L_0})$} }}\sum_{\substack{jb^{L-1}\leq p<(j+1)b^{L-1}\\p\equiv\rev{r}(b^\eta)\\ p\equiv \overline{b}^{L-1}a\ \text{mod $(q,b^2-1)$}\\ \rev{p}\equiv v\thinspace (b^{L_0})}}\log p.
    \end{equation}
    Now, for any $1\leq v<b^{L_0}$ in~\eqref{thirdmaineq} we write
    \begin{equation*}
        v=\sum_{0\leq i<L_0}v_i b^i
    \end{equation*}
    and
    \begin{equation*}
        v^*:=\sum_{0\leq i<L_0}v_i b^{L-i-1}
    \end{equation*}
    for the $L$-digit number in base-$b$ starting with the reversal of $v$ and followed by trailing zeros. In particular, the condition $\rev{p}\equiv v\thinspace(b^{L_0})$ in~\eqref{thirdmaineq} is equivalent to $p$ being in the interval $[v^*,v^*+b^{L-L_0})$. That is,
    \begin{equation}\label{fourthmaineq}
        I_j=\sum_{\substack{1\leq v<b^{L_0}\\ v\equiv j\thinspace (b)\\ v \equiv a\ \text{mod $(q,b^{L_0})$} }}\sum_{\substack{p\in[v^*,v^*+b^{L-L_0})\\p\equiv\rev{r}(b^\eta)\\ p\equiv \overline{b}^{L-1}a\ \text{mod $(q,b^2-1)$}}}\log p.
    \end{equation}
    Since $b^{\eta}\leq L^{\ell}$, the inner sum of~\eqref{fourthmaineq} is readily bounded using the Chinese remainder theorem and the Siegel--Walfisz theorem, yielding
    \begin{align}
        I_j&=\frac{b^{L-L_0}}{\varphi(b^{\eta})\varphi((q,b^2-1))}\sum_{\substack{1\leq v<b^{L_0}\\ v\equiv j\thinspace (b)\\ v \equiv a\ \text{mod $(q,b^{L_0})$}}}1+O_b\left(b^{L+L_0}\exp\left(-c_{SW}\sqrt{L}\right)\right)\label{Ijfinal}
    \end{align}
    for some ineffective constant $c_{\text{SW}}>0$. Substituting~\eqref{Ijfinal} back into~\eqref{firstmaineq} gives
    \begin{align}
        &M(\eta,r,a,q)\\
        &\quad=\frac{(q,(b^2-1)b^L)}{\varphi(b^{\eta})\varphi((q,b^2-1))}\frac{b^{L-L_0}}{q}\sum_{\substack{0\leq j<b\\(j,b)=1}}\sum_{\substack{1\leq v<b^{L_0}\\ v\equiv j\thinspace (b)\\ v \equiv a\ \text{mod $(q,b^{L_0})$}}}1+O_b\left(b^{L+L_0}\exp\left(-c_{SW}\sqrt{L}\right)\right)\notag\\
        &\quad=\frac{(q,(b^2-1)b^L)}{\varphi(b^{\eta})\varphi((q,b^2-1))}\frac{b^{L-L_0}}{q}\sum_{\substack{1\leq v<b^{L_0}\\ (v,b)=1\\ v \equiv a\ \text{mod $(q,b^{L_0})$}}}1+O_b\left(b^{L+L_0}\exp\left(-c_{SW}\sqrt{L}\right)\right).\label{Mneweq}
    \end{align}
    Here, since $b^{L_0}$ is divisible by $(q,b^{L_0})$,
    \begin{equation}\label{isocounteq}
        \sum_{\substack{1\leq v<b^{L_0}\\ (v,b)=1\\ v \equiv a\ \text{mod $(q,b^{L_0})$}}}1=\frac{\varphi(b^{L_0})}{\varphi((q,b^{L_0}))}=\frac{\varphi(b^{L_0})}{\varphi((q,b^{L}))}.
    \end{equation}
    Therefore,
    \begin{align}
        M(\eta,r,a,q)&=\frac{(q,(b^2-1)b^L)}{\varphi(b^{\eta})\varphi((q,b^2-1))}\frac{b^{L-L_0}}{q}\frac{\varphi(b^{L_0})}{\varphi((q,b^L))}+O_b\left(b^{L+L_0}\exp\left(-c_{SW}\sqrt{L}\right)\right)\notag\\
        &=\frac{(q,b^3-b)}{\varphi((q,b^3-b))}\frac{b^{L-\eta}}{q}+O_b\left(b^{L+L_0}\exp\left(-c_{SW}\sqrt{L}\right)\right)\label{finalmeq},
    \end{align}
    where in the second line we used that
    \begin{equation*}
        \frac{\varphi(b^{L_0})}{\varphi(b^{\eta})}=\frac{b^{L_0}}{b^{\eta}}\quad\text{and}\quad\frac{(q,(b^2-1)b^{L})}{\varphi((q,(b^2-1)b^L))}=\frac{(q,b^3-b)}{\varphi((q,b^3-b))},
    \end{equation*}
    which follows upon applying the standard identity $\varphi(n)=n\prod_p(1-1/p)$. To finish, we have to show that the contribution of $b^{L_0}$ to the error term in~\eqref{finalmeq} is small. This is done precisely as on~\cite[p.~27]{bhowmik2024telhcirid} but we repeat the argument here for completeness.

    So, for a prime $p$, let $v_p(n)$ denote the $p$-adic order of $n$. That is, the highest power of $p$ which divides $n$. If $q\leq\exp(c''\sqrt{L})$ for a small constant $c''>0$, then for any prime $p$,
    \begin{equation*}
        v_p(q)\leq\frac{\log q}{\log 2}\leq\frac{c''\sqrt{L}}{\log 2}<L.
    \end{equation*}
    It then follows that
    \begin{equation*}
        (q,b^{\lfloor\frac{c\sqrt{L}}{\log 2}\rfloor})=\prod_{p\mid (q,b)}p^{\min(v_p(q),v_p(b)\lfloor\frac{c\sqrt{L}}{\log 2}\rfloor)}=\prod_{p\mid (q,b)}p^{v_p(q)}=\prod_{p\mid (q,b)}p^{\min(v_p(q),v_p(b)L)}=(q,b^L)
    \end{equation*}
    so that $L_0\leq \frac{c\sqrt{L}}{\log 2}$ by minimality. Thus, provided $c<\frac{1}{2}c_{SW}$ then~\eqref{finalmeq} becomes
    \begin{equation}\label{FinalFinalMeq}
        M(\eta,r,a,q)=\frac{(q,b^3-b)}{\varphi((q,b^3-b))}\frac{b^{L-\eta}}{q}+O_b\left(b^{L}\exp\left(-c''\sqrt{L}\right)\right).
    \end{equation}
    To conclude, we substitute~\eqref{FinalFinalMeq} and~\eqref{errfinal} into~\eqref{thetabigeq}, and incorporate back in the exceptional cases when $(\rev{r},b)>1$ and $(a,q,b^3-b)>1$, giving
    \begin{equation*}
        \rev{\vartheta}_L^*(\eta,r,a,q)=\kappa_b(\rev{r})\rho_b(a,q)\frac{(q,b^3-b)}{\varphi((q,b^3-b))}\frac{b^{L-\eta}}{q} +O_b\bigg(b^L \exp(-c\sqrt{L})\bigg),
    \end{equation*}
    as desired.
\end{proof}

\section{Proof of Theorems~\ref{thm:modifiedZsiflaw} and~\ref{thm: smooth zsiflaw legeis}}\label{smoothsect}
In this section we show how the proofs of Theorems~\ref{thm:modifiedZsiflaw} and~\ref{thm: smooth zsiflaw legeis} follow readily from Proposition~\ref{prop: partitioned Zsiflaw Legeis}, proven in the previous section.

\subsection{Proof of Theorem~\ref{thm:modifiedZsiflaw}}
We begin with the proof of Theorem~\ref{thm:modifiedZsiflaw}.
\begin{proof}[Proof of Theorem~\ref{thm:modifiedZsiflaw}]
    We have
    \begin{equation*}
        \rev{\vartheta}^*_L(a,q):=\sum_{\substack{\rev{p}\in\cB_L^*\\\rev{p}\equiv a\thinspace(\text{mod}\ q)}}\log p=\sum_{1\leq r<b}\sum_{\substack{rb^{L-1}\leq \rev{p}<(r+1)b^{L-1}\\\rev{p}\equiv a(q)\\ (\rev{p},b^3-b)=1}}\log p.
    \end{equation*}
    Hence, applying Proposition~\ref{prop: partitioned Zsiflaw Legeis} with $\eta=1$ gives
    \begin{align*}
        \vartheta_L^*(a,q)&=\sum_{1\leq r<b}\kappa_b(\rev{r})\rho_b(a,q)\frac{(q,b^3-b)}{\varphi((q,b^3-b))}\frac{b^{L-1}}{q}+O_{b}\bigg(b^L \exp(-c\sqrt{L})\bigg)\\
        &=\frac{\varphi(b)}{b}\frac{(q,b^3-b)}{\varphi((q,b^3-b))}\thinspace \frac{\rho_b(a,q)}{q}b^{L}+O_b\left(b^L\exp\left(-c\sqrt{L}\right)\right),
    \end{align*}
    as desired. As discussed in the remark after Proposition~\ref{prop: partitioned Zsiflaw Legeis}, the constant $c$ is effectively computable since $\eta$ was chosen to be a constant.
\end{proof}

\subsection{Proof of Theorem~\ref{thm: smooth zsiflaw legeis}}
In order to prove Theorem~\ref{thm: smooth zsiflaw legeis}, we first prove a slight variation with $x$ taken to be $x=Rb^{L-\eta}$ for some integer $R$ with $b^{\eta-1}<R\leq b^{\eta}$. 
\begin{lemma}\label{lemma: cumulative zsiflaw legeis}
    Keep the same notation and conditions as Theorem~\ref{thm: smooth zsiflaw legeis}. Let $\eta\in\N$ such that $b^{\eta}\leq L^{\ell}$ for any fixed $\ell>0$. Then, for any $b^{\eta-1}< R\leq b^\eta$ we have
    \begin{equation*}
        \rev{\vartheta}^*(Rb^{L-\eta};a,q)=\frac{(q,b^3-b)}{\varphi((q,b^3-b))}\frac{\rho_b(a,q)}{q} |\fB(Rb^{L-\eta})| + O_b\bigg(b^L\exp(-c\sqrt{L})\bigg),
    \end{equation*}
    where $c>0$ is an ineffective constant and depending on $b$.
\end{lemma}
\begin{proof}
    Using the notation~\eqref{thetaetadef}, we first note that
    \begin{equation*}
    \begin{split}
        \rev{\vartheta}^*(Rb^{L-\eta};a,q)=&\sum_{j=1}^{L-1}\sum_{r=1}^{b-1}\rev{\vartheta}_{j}^*(1,r,a,q)+\sum_{b^{\eta-1}\leq r<R}\rev{\vartheta}_L^*(\eta,r,a,q).
    \end{split}
    \end{equation*}
    Applying Proposition~\ref{prop: partitioned Zsiflaw Legeis} then gives
    \begin{equation*}
    \begin{split}
        &\rev{\vartheta}^*(Rb^{L-\eta};a,q)\\
        &\quad=\frac{(q,b^3-b)}{\varphi((q,b^3-b))}\frac{\rho_b(a,q)}{q}\bigg(\sum_{j=1}^{L-1}b^{j-1}\sum_{r=1}^{b-1}\kappa_b(\rev{r})+b^{L-\eta}\sum_{b^{\eta-1}\leq r<R}\kappa_b(\rev{r}) \bigg)\\
        &\qquad\qquad\qquad\qquad\qquad\qquad\qquad\qquad\qquad\qquad\qquad\qquad+O_b\left((L+R)\thinspace b^{L}\exp(-c'\sqrt{L})\right)\\
        &\quad=\frac{(q,b^3-b)}{\varphi((q,b^3-b))}\frac{\rho_b(a,q)}{q} |\fB(Rb^{L-\eta})|+O_b\left((L+R)\thinspace b^{L}\exp(-c'\sqrt{L})\right)
    \end{split}
    \end{equation*}
    where $c'>0$ is an ineffective constant depending on $b$. To finish, we note that $R\leq b^{\eta}\leq L^{\ell}$ so that
    \begin{equation*}
        (L+R)b^L\exp(-c'\sqrt{L})\ll_b b^L\exp(-c\sqrt{L})
    \end{equation*}
    for $c=c'/2$ say.
\end{proof}
We now use Lemma~\ref{lemma: cumulative zsiflaw legeis} to finish the proof of Theorem~\ref{thm: smooth zsiflaw legeis}.
\begin{proof}[Proof of Theorem~\ref{thm: smooth zsiflaw legeis}]
    For some large value of $x$, let $L\in\N$ be the unique integer such that $b^{L-1}< x\leq b^L$. Next, let $\eta\in\N$ be such that 
    \begin{equation}\label{etabounds}
        {(\log x)^{A+1}< b^\eta< L^{A+2}}.
    \end{equation} 
    Furthermore, let $R\in\N$ be the unique integer such that $\frac{x}{b^{L-\eta}}-1\leq R< \frac{x}{b^{L-\eta}}$. Hence, $x-b^{L-\eta}\leq Rb^{L-\eta}<x$ and thus
    \begin{equation*}
            \rev{\vartheta}^*(x;a,q)=\rev{\vartheta}^*(Rb^{L-\eta};a,q)+O_b(Lb^{L-\eta}).
    \end{equation*}
    Applying Lemma~\ref{lemma: cumulative zsiflaw legeis} then gives
    \begin{equation*}
    \begin{split}
        \rev{\vartheta}^*(x;a,q)&= \frac{(q,b^3-b)}{\varphi((q,b^3-b))}\frac{\rho_b(a,q)}{q} |\fB(Rb^{L-\eta})|+O_b(Lb^{L-\eta})\\
        &=\frac{(q,b^3-b)}{\varphi((q,b^3-b))}\frac{\rho_b(a,q)}{q} |\fB(x)|+O_b(Lb^{L-\eta}).
    \end{split}
    \end{equation*}
    We conclude the proof by noting that $Lb^{L-\eta}\ll_{b,A} \frac{x}{(\log x)^A}$ by \eqref{etabounds}.
\end{proof}

\section{Proof of Theorems~\ref{ternaryhcabthm1},~\ref{ternaryhcabthm2} and~\ref{almostthm:intro}}\label{ternsect}
In this section we prove Theorems~\ref{ternaryhcabthm1},~\ref{ternaryhcabthm2} and~\ref{almostthm:intro} using the circle method and our refined Zsiflaw--Legeis theorem (Theorem~\ref{thm: smooth zsiflaw legeis}). Since the proof of each theorem is similar, we focus on the proof of Theorem~\ref{ternaryhcabthm1}, followed by a discussion of the proofs of Theorems~\ref{ternaryhcabthm2} and~\ref{almostthm:intro} at the end of this section (see Subsection~\ref{discsubsect}). From this discussion, it will also be clear as for why one has the conjectural asymptotic for Hcabdlog's conjecture given in Conjecture~\ref{asymhcabcon}. 

The general approach is to modify the standard proof of Vinogradov's three prime theorem (see e.g.~\cite[Chapter~10]{murty2023introduction}), accounting for extra technicalities associated with dealing with reversed primes as opposed to primes.

Before beginning the proof of Theorem~\ref{ternaryhcabthm1}, we first give explicit bounds for the main terms $\cS_{1,2}(N)$ and $\cS_{2,1}(N)$ defined in~\eqref{S12def} and~\eqref{S21def} respectively. This ensures that~\eqref{r12asym} and~\eqref{r21asym} are in fact asymptotic expressions.

\begin{lemma}\label{S12S21lem}
    Let $b\geq 2$ be fixed, and $\cS_{1,2}(N)$ and $\cS_{2,1}(N)$ be as defined in~\eqref{S12def} and~\eqref{S21def} respectively. Then, for sufficiently large odd $N$,
    \begin{equation}\label{S12bounds}
        \frac{N^2}{16b^2}\leq\cS_{1,2}(N)\leq\frac{N^2}{2} 
    \end{equation}
    and
    \begin{equation}\label{S21bounds}
        \frac{N^2}{8b}-\frac{N}{4b}\leq\cS_{2,1}(N)\leq\frac{N^2}{2}. 
    \end{equation}
    In particular, $\cS_{1,2}(N),\cS_{2,1}(N)\asymp_b N^2$.
\end{lemma}
\begin{proof}
    The upper bounds for $\cS_{1,2}(N)$ and $\cS_{2,1}(N)$ follow upon noting that the number of representations of $N$ as the sum of three positive integers is
    \begin{equation*}
        \binom{N-1}{2}=\frac{(N-1)(N-2)}{2}\leq\frac{N^2}{2}.
    \end{equation*}
    For the lower bounds, we use that
    \begin{equation*}
       \cS_{1,2}(N)=\sum_{\substack{n_1,n_2,n_3\leq N\\n_1+n_2+n_3=N\\(\rev{n_2}\rev{n_3},b)=1}}1\geq \sum_{\substack{n_2,n_3\in \fB(N/2)\\ n_1\leq N\\ n_1+n_2+n_3=N}}1 \geq\#\fB\left(\frac{N}{2}\right)^2
    \end{equation*}
    and similarly
    \begin{equation*}
       \cS_{2,1}(N)=\sum_{\substack{n_1,n_2,n_3\leq N\\n_1+n_2+n_3=N\\(\rev{n_3},b)=1}}1\geq \sum_{\substack{n_2\leq N/2,\: n_3\in\fB(N/2)\\n_1\leq N\\n_1+n_2+n_3=N}}1 \geq\#\fB\left(\frac{N}{2}\right)\left\lfloor\frac{N}{2}\right\rfloor.
    \end{equation*}
    To conclude, we then just have to show that \begin{equation}\label{BNbound}
        \#\fB\left(\frac{N}{2}\right)\geq\frac{N}{4b}.
    \end{equation}
    To see why~\eqref{BNbound} holds, let $L\in\N$ be the unique integer with
    \begin{equation}\label{N2bounds}
        2b^L\leq \frac{N}{2}<2b^{L+1}.
    \end{equation}
    Since each of the $b^L$ integers $n$ with $b^L\leq n<2b^L$ satisfy $(\rev{n},b)=1$, it follows that
    \begin{equation}\label{BN2lb}
        \#\fB\left(\frac{N}{2}\right)\geq b^L.
    \end{equation}
    By the upper bound in~\eqref{N2bounds} we have $b^L\geq N/4b$, which substituted into~\eqref{BN2lb} gives~\eqref{BNbound} as desired.
\end{proof}

In the following subsections, we set up and apply the circle method for the proof of Theorem~\ref{ternaryhcabthm1}.

\subsection{Setting up the circle method}\label{setupsub}
With a view to prove Theorem~\ref{ternaryhcabthm1}, we express $\cR_{1,2}(N)$ as
\begin{align}\label{R12inteq}
    \cR_{1,2}(N)&=\sum_{p_1\leq N}\sum_{\substack{\rev{p_2},\rev{p_3}\leq N\\(\rev{p_2}\rev{p_3},b^3-b)=1}}\int_{0}^1 e\left(\alpha(p_1+\rev{p_2}+\rev{p_3}-N)\right)\log p_1\log p_2\log p_3\:\mathrm{d}\alpha\notag\\
    &=\int_0^1S(\alpha)\rev{S}_b^2(\alpha)e(-N\alpha)\mathrm{d}\alpha,
\end{align}
where
\begin{equation*}
    S(\alpha)=S(N,\alpha):=\sum_{p\leq N}e(p\alpha)\log p
\end{equation*}
and
\begin{equation*}
    \rev{S}_b(\alpha)=\rev{S}_b(N,\alpha):=\sum_{\substack{\rev{p}\leq N\\(\rev{p},b^3-b)=1}}e(\rev{p}\alpha)\log p.
\end{equation*}

As in the usual setup for the circle method, we let, for some $B\geq 2$,
\begin{equation}\label{Qdef}
    Q=(\log N)^B.    
\end{equation}
Then, for $1\leq q\leq Q$ and $0\leq a\leq  q$ with $(a,q)=1$, we define the \emph{major arc}
\begin{equation}\label{Maqdef}
    \fM(a,q):=\bigg\{\alpha\in\R: \bigg|\alpha-\frac{a}{q}\bigg|\leq \frac{Q}{N} \bigg\}\cap[0,1].
\end{equation}
For sufficiently large $N$, the major arcs are disjoint and we write
\begin{equation*}
    \fM:=\bigcup_{q\leq Q}\bigcup_{\substack{a\leq q\\(a,q)=1}}\fM(a,q)
\end{equation*}
for the \emph{set of major arcs}, and
\begin{equation*}
    \fm:=[0,1]\setminus\fM
\end{equation*}
for the \emph{set of minor arcs}. Our goal now is to use Theorem~\ref{thm: smooth zsiflaw legeis} to extract the desired asymptotic for $\cR_{1,2}(N)$ from the contribution of the major arcs in~\eqref{R12inteq}, whilst also showing that the contribution from the minor arcs is asymptotically small.

\subsection{Some exponential sum bounds and identities}
Before diving into our major and minor arc estimates, we require several results on exponential sums. The first of these is usually called \emph{Ramanujan's sum}.

\begin{lemma}[{See e.g.~\cite[p.~164]{apostol1976introduction}}]\label{ramlem}
    For any $a,q\in\N$, let
    \begin{equation}\label{ramdef}
        c_q(a):=\sum_{{(r,q)=1}}e(ra/q).
    \end{equation}
    Then,
    \begin{equation*}
        c_q(a)=\frac{\mu\left(\frac{q}{(a,q)}\right)\varphi(q)}{\varphi\left(\frac{q}{(a,q)}\right)}.
    \end{equation*}
    In particular, if $(a,q)=1$ then $c_q(a)=\mu(q)$.
\end{lemma}
The next result allows us to handle the factor $\rho_b(a,q)$ from Theorem~\ref{thm: smooth zsiflaw legeis}.
\begin{lemma}\label{lemma: exponential sum}
    Let $b\geq 2$, $q\in\N$, $(a,q)=1$, and $\rho_b(r,q)$ be as defined in~\eqref{rhobdef}. Then,
    \begin{equation}
        \sum_{r\thinspace(q)}e(ra/q)\rho_b(r,q)=\mu(q)\mathbb{1}_{q\mid b^3-b},
    \end{equation}
    where
    \begin{equation}\label{bb1def}
        \mathbb{1}_{q\mid b^3-b}=\begin{cases}
            1 & \text{if } q\mid b^3-b,\\
            0 & \text{otherwise.}
        \end{cases}
    \end{equation}
\end{lemma}

\begin{proof}
    In what follows, let
    \begin{equation*}
        f_b(q):=\sum_{r\thinspace(q)}e(ra/q)\rho_b(r,q)=\sum_{\substack{r\thinspace(q)\\(r,q,b^3-b)=1}}e(ra/q)
    \end{equation*}
    for convenience. We prove the lemma by considering three different cases.\\
    \\
    \textbf{Case 1:} $(q,b^3-b)=1$.\\
    In this case, we always have $\rho_b(r,q)=1$ so that
    \begin{equation*}
        f_b(q)=\sum_{r\thinspace (q)}e(ra/q)=
        \begin{cases}
            1,&\text{if $q=1$},\\
            0,&\text{otherwise.}
        \end{cases}
    \end{equation*}
    \textbf{Case 2:} $(q,b^3-b)>1$ and $q\mid b^3-b$.\\
    In this case, we always have $(q,b^3-b)=q$. Thus,
    \begin{equation*}
        f_b(q)=\sum_{\substack{r\thinspace(q)\\(r,q,b^3-b)=1}}e(ra/q)=\sum_{\substack{r\thinspace(q)\\(r,q)=1}}e(ra/q)=\mu(q),
    \end{equation*}
    where in the last equality we applied Lemma~\ref{ramlem}.\\
    \\
    \textbf{Case 3:} $(q,b^3-b)>1$ and $q\nmid b^3-b$.\\
    We begin by writing
    \begin{equation*}
        f_b(q)=\sum_{\substack{r\thinspace(q)\\(r,q,b^3-b)=1}}e(ra/q)=\sum_{r\thinspace(q)}e(ra/q)-\sum_{\substack{r\thinspace(q)\\(r,q,b^3-b)>1}}e(ra/q)=-\sum_{\substack{r\thinspace(q)\\(r,q,b^3-b)>1}}e(ra/q).
    \end{equation*}
    Splitting this sum by the value of $(r,q,b^3-b)$ gives
    \begin{equation}\label{case3eq1}
        f_b(q)=-\sum_{\substack{r\thinspace(q)\\(r,q,b^3-b)>1}}e(ra/q)=-\sum_{\substack{d\mid (q,b^3-b)\\d\neq 1}}\sum_{\substack{r\thinspace(q)\\(r,q,b^3-b)=d}}e(ra/q).
    \end{equation}
    Consider the inner sum of \eqref{case3eq1} for some fixed $d=d_1$. Writing $r_1=r/d$ and $q_1=q/d$ gives
    \begin{equation*}
        \sum_{\substack{r\thinspace(q)\\(r,q,b^3-b)=d}}e(ra/q)=\sum_{\substack{r_1\thinspace(q_1)\\\left(r_1,q_1,\frac{b^3-b}{d_1}\right)=1}}e(r_1a/q_1).
    \end{equation*}
    Here, it is vital to note that since $q\nmid b^3-b$, one has $q_1\neq 1$. In particular, $\sum_{r_1\thinspace(q_1)}e(r_1a/q_1)=0$. Thus, by an analogous procedure to before,
    \begin{align*}
        \sum_{\substack{r_1\thinspace(q_1)\\\left(r_1,q_1,\frac{b^3-b}{d_1}\right)=1}}e(r_1a/q_1)&=-\sum_{\substack{r_1\thinspace(q_1)\\\left(r_1,q_1,\frac{b^3-b}{d_1}\right)>1}}e(r_1a/q_1)\\
        &=-\sum_{\substack{d_1d_2\mid \left(q,b^3-b\right)\\d_2\neq 1}}\sum_{\substack{r_2\thinspace(q_2)\\\left(r_2,q_2,\frac{b^3-b}{d_1d_2}\right)=1}}e(r_2a/q_2),
    \end{align*}
    with $q_2=q/d_1d_2\neq 0$. Putting everything together,
    \begin{equation*}
        f_b(q)=\sum_{\substack{d_1,d_2\\d_1d_2\mid \left(q,b^3-b\right)\\d_2\neq 1}}\sum_{\substack{r_2\thinspace(q_2)\\\left(r_2,q_2,\frac{b^3-b}{d_1d_2}\right)=1}}e(r_2a/q_2).
    \end{equation*}
    Repeating the same procedure again then gives
    \begin{equation*}
        f_b(q)=-\sum_{\substack{d_1,d_2,d_3\\d_1d_2d_3\mid \left(q,b^3-b\right)\\d_1,d_2,d_3\neq 1}}\sum_{\substack{r_3\thinspace(q_3)\\\left(r_3,q_3,\frac{b^3-b}{d_1d_2d_3}\right)=1}}e(r_3a/q_3)
    \end{equation*}
    and so forth. In particular, if we let
    \begin{equation}\label{kdef}
        k=\Omega((q,b^3-b))
    \end{equation} 
    be the total number of prime factors of $(q,b^3-b)$, then
    \begin{equation}\label{case3eq2}
        f_b(q)=(-1)^k\sum_{\substack{d_1,\ldots,d_k\\d_1\cdots d_k\mid \left(q,b^3-b\right)\\d_1,\ldots,d_k\neq 1}}\sum_{\substack{r_k\thinspace(q_k)\\\left(r_k,q_k,\frac{b^3-b}{d_1\cdots d_k}\right)=1}}e(r_ka/q_k),
    \end{equation}
    where $q_k=q/(d_1\cdots d_k)\neq 1$. By the choice \eqref{kdef} of $k$, we have
    \begin{equation*}
        \left(q_k,\frac{b^3-b}{d_1\cdots d_k}\right)=\left(\frac{q}{d_1\cdots d_k},\frac{b^3-b}{d_1\cdots d_k}\right)=1.
    \end{equation*}
    As a result, the inner sum in \eqref{case3eq2} vanishes and $f_b(q)=0$ as required.
\end{proof}

We now give an identity related to the \emph{singular series} $\fS_3(N)$ defined in~\eqref{sing1eq}.

\begin{lemma}\label{singlem}
Let $c_q(N)$ be as in Lemma~\ref{ramlem}. Then, for all $b\geq 2$ and $N\in\mathbb{Z}$,
\begin{equation*}
    \sum_{\substack{q=1\\q\mid b^3-b}}^\infty\frac{\mu(q)}{\varphi(q)^3}c_q(N)=\prod_{\substack{p\mid b^3-b\\ p\mid N}}\left(1-\frac{1}{(p-1)^2}\right)\prod_{\substack{p\mid b^3-b\\ p\nmid N}}\left(1+\frac{1}{(p-1)^3}\right)=:\fS_3(N).
\end{equation*}
\end{lemma}
\begin{proof}
    By converting to an Euler product expansion,
    \begin{equation}\label{singeulereq}
        \sum_{\substack{q=1\\q\mid b^3-b}}^\infty\frac{\mu(q)}{\varphi(q)^3}c_q(N)=\prod_{p\mid b^3-b}\left(1-\frac{c_p(N)}{(p-1)^3}\right).
    \end{equation}
    Now, by Lemma~\ref{ramlem} we have that for any prime $p$
    \begin{equation*}
        c_p(N)=
        \begin{cases}
            p-1,&\text{if $p\mid N$},\\
            -1,&\text{otherwise.}
        \end{cases}
    \end{equation*}
    Thus, by splitting the product~\eqref{singeulereq} into the cases where $p\mid N$ and $p\nmid N$, we obtain the desired result.
\end{proof}

Next we give two variants of~\emph{Weyl's inequality}, which will be used to extract the factor $\cS_{1,2}(N)$ defined in~\eqref{S12def}.

\begin{lemma}[{See e.g.\ \cite[Lemma~8.2]{murty2023introduction}}]\label{Weyllem}
    For $\beta\in\R$ and $N_1<N_2$,
    \begin{equation*}
        \left|\sum_{m=N_1+1}^{N_2}e(m\beta)\right|\ll\frac{1}{||\beta||}.
    \end{equation*}
\end{lemma}

\begin{lemma}\label{Weylrevlem}
    For $\beta\in\R$ and any fixed $b\geq 2$,
    \begin{equation*}
        \left|\sum_{\substack{m\leq N\\(\rev{m},b)=1}}e(m\beta)\right|\ll_b\frac{\log N}{||\beta||}
    \end{equation*}
\end{lemma}
\begin{proof}
    Let $L$ be the number of digits of $N$ in base $b$ and $R$ be the unique integer such that
    \begin{equation*}
        Rb^{L-1}\leq N<(R+1)b^{L-1}.
    \end{equation*}
    Then, by the triangle inequality and Lemma~\ref{Weyllem},
    \begin{align*}
        \left|\sum_{\substack{m\leq N\\(\rev{m},b)=1}}e(m\beta)\right|&\leq\sum_{\ell=0}^{L-2}\sum_{\substack{(r,b)=1\\0\leq r<b}}\left|\sum_{rb^{\ell}\leq m<(r+1)b^{\ell}}e(m\beta)\right|+\sum_{\substack{r\leq R\\(r,b)=1}}\left|\sum_{rb^{L-1}\leq m<(r+1)b^{L-1}}e(m\beta)\right|\\
        &\leq \sum_{\ell=0}^{L-2}\sum_{\substack{(r,b)=1\\0\leq r<b}}\frac{1}{||\beta||}+\sum_{\substack{r\leq R\\(r,b)=1}}\frac{1}{||\beta||}\\
        &\ll_b\frac{L}{||\beta||},
    \end{align*}
    from which the desired bound follows, since $L\asymp_b\log N$.
\end{proof}

The final bound we require is Vinogradov's classical minor arc estimate for exponential sums over primes.

\begin{lemma}[{See e.g.\ \cite[Theorem~8.16]{murty2023introduction}}]\label{vinlem}
    Suppose $\alpha$ is a real number satisfying
    \begin{equation*}
        \left|\alpha-\frac{a}{q}\right|\leq\frac{1}{q^2}
    \end{equation*}
    for some rational $a/q$ where $1\leq q\leq N$ and $(a,q)=1$. Then,
    \begin{equation*}
        S(\alpha):=\sum_{p\leq N}e(p\alpha)\log p\ll\left(\frac{N}{\sqrt{q}}+N^{4/5}+N^{1/2}q^{1/2}\right)(\log N)^4.
    \end{equation*}
\end{lemma}

\subsection{The major arcs}\label{majorsect}
We now deal with the major arc component of~\eqref{R12inteq}. That is,
\begin{equation*}
    \int_{\fM}S(\alpha)\rev{S}_b^2(\alpha)e(-N\alpha)\mathrm{d}\alpha.
\end{equation*}
To do so, we first obtain asymptotic expressions for $S(\alpha)$ and $\rev{S}_b(\alpha)$ (as $N\to\infty$) when $\alpha$ lies on a major arc $\fM(a,q)$. In what follows, we let
\begin{align*}
    v(\beta):=\sum_{n\leq N}e(n\beta),\qquad \rev{v}_b(\beta):=\sum_{\substack{n\leq N\\(\rev{n},b)=1}}e(n\beta)
\end{align*}
and
\begin{equation*}
    \E(q):=\frac{\mu(q)}{\varphi(q)},\qquad\ \rev{\E}_b(q):=\frac{\mu(q)}{\varphi(q)}\mathbb{1}_{q\mid b^3-b}
\end{equation*}
where $\mathbb{1}_{q\mid b^3-b}$ is as in~\eqref{bb1def}. With this notation, one has the following well-known expression for $S(\alpha)$.
\begin{lemma}[{See e.g.~\cite[Theorem~10.12]{murty2023introduction}}]\label{salphlem}
    Let $\alpha\in\fM(a,q)$ and $\beta=\alpha-a/q$. Then for any $C>0$, 
    \begin{equation*}
        S(\alpha)=\E(q)v(\beta)+O_{B,C}\left(\frac{N}{(\log N)^C}\right),
    \end{equation*}
    with $B\geq 1$ as in the definition (\eqref{Qdef} and~\eqref{Maqdef}) of $\fM(a,q)$. 
\end{lemma}

Using Lemma~\ref{lemma: exponential sum} and~Theorem~\ref{thm: smooth zsiflaw legeis}, we then obtain the following analogous expression for $\rev{S}_b(\alpha)$.

\begin{lemma}\label{revsalphalem}
    Let $\alpha\in\fM(a,q)$ and $\beta=\alpha-a/q$. Then for any $C>0$,
    \begin{equation}\label{revSexp}
        \rev{S}_b(\alpha)=\rev{\E}_b(q)\rev{v}_b(\beta)+O_{b,B,C}\left(\frac{N}{(\log N)^C}\right),
    \end{equation}
    with $B\geq 1$ as in the definition (\eqref{Qdef} and~\eqref{Maqdef}) of $\fM(a,q)$.
\end{lemma}
\begin{proof}
    First we prove that
    \begin{equation}\label{Sxaqdef}
        \rev{S}_b(x,a/q):=\sum_{\substack{\rev{p}\leq x\\(\rev{p},b^3-b)=1}}e(\rev{p}a/q)\log p
    \end{equation}
    satisfies
    \begin{equation}\label{Sxaqeq}
        \rev{S}_b(x,a/q)=\rev{\E}_b(q)\#\fB(x)+O_{b,B,C}\left(\frac{x}{(\log x)^{B+C}}\right).
    \end{equation}
    To see why~\eqref{Sxaqeq} holds, we begin by splitting the sum in~\eqref{Sxaqdef} based on the residue class of $\rev{p}$ mod $q$:
    \begin{align*}
        \rev{S}_b(x,a/q)=\sum_{r\thinspace(q)}\sum_{\substack{\rev{p}\leq x\\ \rev{p}\equiv r\thinspace(q)\\(\rev{p},b^3-b)=1}}e(\rev{p}a/q)\log p=\sum_{r\thinspace(q)}e(ra/q)\sum_{\substack{\rev{p}\leq x\\ \rev{p}\equiv r\thinspace(q)\\(\rev{p},b^3-b)=1}}\log p.
    \end{align*}
    Applying Theorem~\ref{thm: smooth zsiflaw legeis} with $A=2B+C$ then gives
    \begin{equation*}
        \rev{S}_b(x,a/q)=\frac{(q,b^3-b)}{\varphi((q,b^3-b))} \frac{\#\fB(x)}{q} \sum_{r(q)}e(ra/q)\rho_b(a,q)+ O_{b,B,C}\bigg(\frac{qx}{(\log  x)^{2B+C}}\bigg).
    \end{equation*}
    From here, we then apply Lemma~\ref{lemma: exponential sum} and the fact that $q\leq Q=(\log x)^B$ to yield
    \begin{align*}
        \rev{S}_b(x,a/q)&=\left(\frac{(q,b^3-b)}{\varphi((q,b^3-b))} \frac{\mu(q)}{q}\mathbb{1}_{q\mid b^3-b}\right)\#\fB(x)+ O_{b,B,C}\bigg(\frac{x}{(\log  x)^{B+C}}\bigg)\\
        &=\rev{\E}_b(q)\#\fB(x)+O_{b,B,C}\left(\frac{x}{(\log x)^{B+C}}\right)
    \end{align*}
    as claimed. We now prove~\eqref{revSexp}. In what follows, let
    \begin{equation*}
        \rev{\Lambda}'_b(n)=
        \begin{cases}
            \log p,&\text{if $\rev{n}=p$ is prime and $(n,b^3-b)=1$},\\
            0,&\text{otherwise.}
        \end{cases}
    \end{equation*}
    Then,
    \begin{align*}
        \rev{S}_b(\alpha)=\sum_{\substack{n\leq x}}e(n\alpha)\rev{\Lambda}_b'(n)=\sum_{\substack{n\leq x\\(\rev{n},b)=1}}e(n\alpha)\rev{\Lambda}_b'(n)+O_b(1),
    \end{align*}
    where we have used that there are $O_b(1)$ primes $p$ with $(p,b)>1$. Consequently,
    \begin{align*}
        \rev{S}_b(\alpha)-\rev{\E}_b(q)\rev{v}_b(\beta)&=\sum_{\substack{n\leq x\\(\rev{n},b)=1}}e(n\alpha)\rev{\Lambda}_b'(n)-\rev{\E}_b(q)\sum_{\substack{n\leq x\\(\rev{n},b)=1}}e(n\beta)+O_b(1)\\
        &=\sum_{\substack{n\leq x\\(\rev{n},b)=1}}\left(e(na/q)\rev{\Lambda}_b'(n)-\rev{\E}_b(q)\right)e(n\beta)+O_b(1).
    \end{align*}
    Now, by~\eqref{Sxaqdef}, 
    \begin{equation*}
        A(x):=\sum_{\substack{n\leq x\\(\rev{n},b)=1}}\left(e(na/q)\rev{\Lambda}_b'(n)-\rev{\E}_b(q)\right)=O_{b,B,C}\left(\frac{x}{(\log x)^{B+C}}\right).
    \end{equation*}
    Therefore, since $|\beta|\leq Q/N$ (by~\eqref{Maqdef}), partial summation gives
    \begin{align*}
        \rev{S}_b(\alpha)-\rev{\E}_b(q)\rev{v}_b(\beta)&=A(N)e(\beta N)-\beta\int_1^NA(x)e(\beta x)\mathrm{d}x\\
        &\ll |A(N)|+|\beta|N\max_{1\leq x\leq N}|A(x)|\\
        &\ll_{b,B,C}\frac{QN}{(\log N)^{B+C}}\\
        &\ll\frac{N}{(\log N)^C}, 
    \end{align*}
    as required.
\end{proof}

Using Lemmas~\ref{salphlem} and~\ref{revsalphalem}, we now give the major arc contribution, which matches the expected asymptotic~\eqref{r12asym}.
\begin{proposition}\label{majorprop}
    We have,
    \begin{equation}\label{majorpropeq}
        \int_{\fM}S(\alpha)\rev{S}_b^2(\alpha)e(-N\alpha)\mathrm{d}\alpha=\mathfrak{S}_3(N)\cS_{1,2}(N)+O_{b,B}\left(\frac{N^2}{(\log N)^B}\right),
    \end{equation}
    where $\mathfrak{S}_3(N)$ and $\cS_{1,2}(N)$ are as in~\eqref{sing1eq} and~\eqref{S12def} respectively.
\end{proposition}
\begin{proof}
    By Lemmas~\ref{salphlem} and~\ref{revsalphalem},
    \begin{equation*}
        S(\alpha)\rev{S}_b^2(\alpha)=\E(q)(\rev{\E}_b(q))^2\:v\left(\alpha-\frac{a}{q}\right)\left(\rev{v}_b\left(\alpha-\frac{a}{q}\right)\right)^2+O_{b,B,C}\left(\frac{N^3}{(\log N)^C}\right)
    \end{equation*}
    for $\alpha\in\fM(a,q)$.
    Therefore,
    \begin{align}
        &\int_{\fM}S(\alpha)\rev{S}^2_b(\alpha)e(-N\alpha)\mathrm{d}\alpha\notag\\
        &=\sum_{q=1}^Q\sum_{\substack{0\leq a\leq q\\(a,q)=1}}\E(q)(\rev{\E}_b(q))^2\notag\\
        &\qquad\times\int_{\fM(a,q)}\left(v\left(\alpha-\frac{a}{q}\right)\left(\rev{v}_b\left(\alpha-\frac{a}{q}\right)\right)^2 e(-N\alpha)+O_{b,B,C}\left(\frac{N^3}{(\log N)^C}\right)\right)\mathrm{d}\alpha.\label{bigmajoreq}
    \end{align}
    For the big-$O$ term in~\eqref{bigmajoreq}, we note that the Lebesgue measure of each major arc is $\mu(\fM(a,q))\ll Q/N$ and thus 
    \begin{equation*}
        \sum_{q=1}^Q\sum_{\substack{0\leq a<q\\(a,q)=1}}\int_{\fM(a,q)}\frac{N^3}{(\log N)^C}\mathrm{d}\alpha\ll\frac{Q^3N^2}{(\log N)^C}\ll\frac{N^2}{(\log N)^{C-3B}}.
    \end{equation*}
    That is, provided $C$ is chosen large enough $(C\geq 4B)$, the big-$O$ term in~\eqref{bigmajoreq} can be absorbed into the desired error term in~\eqref{majorpropeq}. We thus focus on the ``main" term in~\eqref{bigmajoreq}. Here,
    \begin{align}
        &\sum_{q=1}^Q\sum_{\substack{0\leq a\leq q\\(a,q)=1}}\E(q)(\rev{\E}_b(q))^2\int_{\fM(a,q)}v\left(\alpha-\frac{a}{q}\right)\left(\rev{v}_b\left(\alpha-\frac{a}{q}\right)\right)^2 e(-N\alpha)\mathrm{d}\alpha\notag\\
        &\quad=\sum_{\substack{1\leq q\leq Q\\q\mid b^3-b}}\frac{\mu(q)}{\varphi(q)^3}\sum_{\substack{1\leq a\leq q\\(a,q)=1}}e\left(-\frac{Na}{q}\right)\int_{\frac{a}{q}-Q/N}^{\frac{a}{q}+Q/N}v\left(\alpha-\frac{a}{q}\right)\left(\rev{v}_b\left(\alpha-\frac{a}{q}\right)\right)^2e\left(-N\left(\alpha-\frac{a}{q}\right)\right)\mathrm{d}\alpha\notag\\
        &\quad=\sum_{\substack{1\leq q\leq Q\\q\mid b^3-b}}\frac{\mu(q)}{\varphi(q)^3}c_{q}(-N)\int_{-Q/N}^{Q/N}v(\beta)(\rev{v}_b(\beta))^2e(-N\beta)\mathrm{d}\beta,\label{majormaineq}
    \end{align}
    where $c_q(-N)$ is as defined in Lemma~\ref{ramlem}. We simplify~\eqref{majormaineq} in two steps. First, we consider the sum in front of the integral. Here, Lemma~\ref{singlem} gives
    \begin{equation}\label{singapproxeq}
        \sum_{\substack{1\leq q\leq Q\\q\mid b^3-b}}\frac{\mu(q)}{\varphi(q)^3}c_q(-N)=\fS_3(N)-\sum_{\substack{q>Q\\q\mid b^3-b}}\frac{\mu(q)}{\varphi(q)^3}c_q(-N).
    \end{equation}
    If $N$ and thus $Q=(\log N)^B$ is sufficiently large, then the tail sum in~\eqref{singapproxeq} contains no terms. That is,
    \begin{equation}\label{majfact1}
        \sum_{\substack{1\leq q\leq Q\\q\mid b^3-b}}\frac{\mu(q)}{\varphi(q)^3}c_q(-N)=\fS_3(N)
    \end{equation}
    for sufficiently large $N$. As for the integral in~\eqref{majormaineq}, Lemmas~\ref{Weyllem} and~\ref{Weylrevlem} give
    \begin{equation*}
        v(\beta)\ll\frac{1}{||\beta||}\quad\text{and}\quad \rev{v}_b(\beta)\ll_b\frac{\log N}{||\beta||}
    \end{equation*}
    so that
    \begin{equation*}
        \int_{Q/N}^{1/2}v(\beta)(\rev{v}_b(\beta))^2e(-N\beta)\mathrm{d}\beta\ll_b \int_{Q/N}^{1/2}\frac{(\log N)^2}{\beta^3}\mathrm{d}\beta\ll\frac{N^2(\log N)^2}{Q^2}\ll \frac{N^2}{(\log N)^B}
    \end{equation*}
    and similarly
    \begin{equation*}
        \int_{-1/2}^{-Q/N}v(\beta)(\rev{v}_b(\beta))^2e(-N\beta)\mathrm{d}\beta\ll_b\frac{N^2}{(\log N)^B}.
    \end{equation*}
    Therefore,
    \begin{align}
        \int_{-Q/N}^{Q/N}v(\beta)(\rev{v}_b(\beta))^2e(-N\beta)\mathrm{d}\beta&=\int_{-1/2}^{1/2}v(\beta)(\rev{v}_b(\beta))^2e(-N\beta)\mathrm{d}\beta+O_b\left(\frac{N^2}{(\log N)^B}\right)\notag\\
        &=\cS_{1,2}(N)+O_b\left(\frac{N^2}{(\log N)^B}\right).\label{majfact2}
    \end{align}
    Substituting~\eqref{majfact1} and~\eqref{majfact2} into~\eqref{majormaineq} and then~\eqref{bigmajoreq}, we obtain
    \begin{equation*}
        \int_{\fM}S(\alpha)\rev{S}_b^2(\alpha)e(-N\alpha)\mathrm{d}\alpha=\mathfrak{S}_3(N)\cS_{1,2}(N)+O_{b,B}\left(\frac{N^2}{(\log N)^B}\right)
    \end{equation*}
    as desired.
\end{proof}

\subsection{The minor arcs and concluding the proof of Theorem~\ref{ternaryhcabthm1}}\label{minorsect}
To conclude the proof of Theorem~\ref{ternaryhcabthm1}, we now deal with the minor arc contribution. That is,
\begin{equation*}
    \int_{\fm}S(\alpha)\rev{S}_b^2(\alpha)e(-N\alpha)\mathrm{d}\alpha.
\end{equation*}

Given Vinogradov's exponential sum bound (Lemma~\ref{vinlem}), bounding the minor arc contribution becomes a relatively straightforward task. The only additional lemma we need is the following $L^2$ estimate.
\begin{lemma}\label{parslem}
    One has
    \begin{equation*}
        \int_0^1|\rev{S}_b(\alpha)|^2\mathrm{d}\alpha\ll_b N\log N.
    \end{equation*}
\end{lemma}
\begin{proof}
    Let $L$ be the digit length of $N$ in base-$b$, so that
    \begin{equation*}
        b^{L-1}\leq N<b^L.
    \end{equation*}
    By the prime number theorem and partial summation,
    \begin{equation*}
        \sum_{p\leq b^L}(\log p)^2\sim b^L\log b^L\ll_bN\log N.
    \end{equation*}
    Therefore,
    \begin{align*}
        \int_0^1|\rev{S}_b(\alpha)|^2\mathrm{d}\alpha&=\int_0^1\sum_{\rev{p_1},\rev{p_2}\leq N}e((\rev{p_1}-\rev{p_2})\alpha)\log p_1\log p_2\:\mathrm{d\alpha}\\
        &=\sum_{\rev{p_1},\rev{p_2}\leq N}\int_0^1e((\rev{p_1}-\rev{p_2})\alpha)\log p_1\log p_2\:\mathrm{d\alpha}\\
        &\ll\sum_{\rev{p}\leq N}(\log p)^2\ll\sum_{\rev{p}\leq b^L}(\log p)^2\ll N\log N,
    \end{align*}
    as required.
\end{proof}
Our minor arc estimate is then as follows.
\begin{proposition}\label{minorprop}
    We have
    \begin{equation*}
        \int_{\fm}S(\alpha)\rev{S}_b^2(\alpha)e(-N\alpha)\mathrm{d}\alpha\ll_{b,B}\frac{N}{(\log N)^{\frac{B}{2}-5}}.
    \end{equation*}
\end{proposition}
\begin{proof}
    With a view to apply Vinogradov's bound (Lemma~\ref{vinlem}), let $\alpha\in\fm$ be arbitrary. By Dirichlet's approximation theorem, there exists integers $a,q$ with $(a,q)=1$ and $1\leq q\leq N/Q$ such that
    \begin{equation*}
        \left|\alpha-\frac{a}{q}\right|\leq\frac{Q}{qN}.
    \end{equation*}
    In particular, 
    \begin{equation}\label{alphaineq1}
        \left|\alpha-\frac{a}{q}\right|\leq\frac{Q}{N}
    \end{equation}
    and
    \begin{equation}\label{alphaineq2}
        \left|\alpha-\frac{a}{q}\right|\leq\frac{1}{q^2}.
    \end{equation}
    Since $\alpha\in\fm$ we must have $Q< q\leq N/Q$ for otherwise~\eqref{alphaineq1} implies that ${\alpha\in\fM(a,q)\subseteq\fM}$. In addition,~\eqref{alphaineq2} means that we may apply Lemma~\ref{vinlem} with such a choice of $q$. Thus, by Lemmas~\ref{vinlem},~\ref{parslem} and the above considerations, \begin{align*}
        \int_{\fm}S(\alpha)\rev{S}_b^2(\alpha)e(-N\alpha)\mathrm{d}\alpha&\leq \sup_{\fm}|S(\alpha)|\int_{0}^1|\rev{S}_b(\alpha)|^2\mathrm{d}\alpha\\
        &\ll_{b}\left(\frac{N}{\sqrt{Q}}+N^{4/5}+\frac{N}{\sqrt{Q}}\right)N(\log N)^5\\
        &\ll_{b}\frac{N^2}{(\log N)^{\frac{B}{2}-5}}
    \end{align*}
    as required.
\end{proof}

Combining all of our results, the proof of Theorem~\ref{ternaryhcabthm1} readily follows.

\begin{proof}[Proof of Theorem~\ref{ternaryhcabthm1}]
    By Propositions~\ref{majorprop} and~\ref{minorprop},
    \begin{align*}
        \int_{0}^1S(\alpha)\rev{S}_b^2(\alpha)e(-N\alpha)\mathrm{d}\alpha&=\int_{\fM}S(\alpha)\rev{S}_b^2(\alpha)e(-N\alpha)\mathrm{d}\alpha+\int_{\fm}S(\alpha)\rev{S}_b^2(\alpha)e(-N\alpha)\mathrm{d}\alpha\\
        &=\mathfrak{S}_3(N)\cS_{1,2}(N)+O_{b,B}\left(\frac{N}{(\log N)^{\frac{B}{2}-5}}\right).
    \end{align*}
    We conclude by setting $A=B/2-5$ for $B>10$.
\end{proof}

\subsection{Discussion of the proofs of Theorems~\ref{ternaryhcabthm2} and~\ref{almostthm:intro}}\label{discsubsect}
In this subsection we describe how one can modify the proof of Theorem~\ref{ternaryhcabthm1} to prove Theorems~\ref{ternaryhcabthm2} and~\ref{almostthm:intro}. In particular, the proof of all three results follows a similar framework using Theorem~\ref{thm: smooth zsiflaw legeis} and the circle method.

\subsubsection{{The proof of Theorem~\ref{ternaryhcabthm2}}}
For Theorem~\ref{ternaryhcabthm2}, we need to find an asymptotic for
\begin{equation*}
    \cR_{2,1}(N)=\int_0^1S(\alpha)^2\rev{S}_b(\alpha)e(-N\alpha)\mathrm{d}\alpha.
\end{equation*}
To do so, one again splits up the unit interval $[0,1]$ into major and minor arcs as discussed in Subsection~\ref{setupsub}. Extracting the main term from the major arcs is done in the exact same way as in the proof of Theorem~\ref{ternaryhcabthm1}, yielding the asymptotic~\eqref{r21asym} in place of~\eqref{r12asym}. 

The argument for the minor arcs is slightly different though. Here, the Cauchy--Schwarz inequality gives
\begin{align*}
    \int_{\fm}S(\alpha)^2\rev{S}_b(\alpha)e(-N\alpha)\mathrm{d}\alpha&\leq\sup_{\alpha\in\fm}|S(\alpha)|\int_{0}^1\left|S(\alpha)\rev{S}_b(\alpha)\right|\mathrm{d}\alpha\\
    &\leq\sup_{\alpha\in\fm}|S(\alpha)|\left(\int_0^1|S(\alpha)|^2\mathrm{d}\alpha\int_0^1|\rev{S}_b(\alpha)|^2\mathrm{d}\alpha\right)^{1/2}.
\end{align*}
To suitably bound this expression, one can use Vinogradov's bound (Lemma~\ref{vinlem}) to handle the $\sup_{\alpha\in\fm}|S(\alpha)|$ factor and Lemma~\ref{parslem} to handle $\int_{\alpha\in\fm}|\rev{S}_b(\alpha)|^2\mathrm{d}\alpha$. As for the integral over $|S(\alpha)|^2$, one can use the standard bound
\begin{equation}\label{parseq2}
    \int_0^1|S(\alpha)|^2\mathrm{d}\alpha\ll N\log N,
\end{equation}
which is analogous to Lemma~\ref{parslem} and proven in the same way.

\subsubsection{{The proof of Theorem~\ref{almostthm:intro}}}
For Theorem~\ref{almostthm:intro} we use a dyadic decomposition. That is, it suffices to show that there are at most $O_{A,b}(x/(\log x)^A)$ even $N$ with $x\leq N\leq 2x$ and  
\begin{equation*}
    N\neq p_1+\rev{p_2}.
\end{equation*}
For this, we show that
\begin{equation}\label{meanvalueeq}
    \sum_{\substack{x\leq N\leq 2x\\2\mid N}}\left|\cR_{1,1}(N)-\fS_2(N)\#\fB(N)\right|^2\ll_{b,A}\frac{x^3}{(\log x)^{A+1}}
\end{equation}
where $\cR_{1,1}(N)$ and $\fS_2(N)$ are as defined in~\eqref{R11def} and~\eqref{sing2eq} respectively. To see why~\eqref{meanvalueeq} suffices, suppose that there were $\gg_{b,A} x/(\log x)^A$ values of $N\in[x,2x]$ with $\cR_{1,1}(N)=0$. Since $\#\fB(N)\gg N$ (see~\eqref{BNbound}) one would then have
\begin{equation*}
    \sum_{\substack{x\leq N\leq 2x\\2\mid N}}\left|\cR_{1,1}(N)-\fS_2(N)\#\fB(N)\right|^2\gg_{b,A}\frac{x^3}{(\log x)^{A}},
\end{equation*}
which contradicts~\eqref{meanvalueeq}.

We now discuss how to prove~\eqref{meanvalueeq} using the circle method. So, analogous to the proofs of Theorem~\ref{ternaryhcabthm1} and~\ref{ternaryhcabthm2}, we write
\begin{align*}
    \cR_{1,1}(N)&=\int_{0}^1S(\alpha)\rev{S}_b(\alpha)e(-N\alpha)\mathrm{d}\alpha\\
    &=\int_\fM S(\alpha)\rev{S}_b(\alpha)e(-N\alpha)\mathrm{d}\alpha+\int_{\fm}S(\alpha)\rev{S}_b(\alpha)e(-N\alpha)\mathrm{d}\alpha
\end{align*}
with $\fM$ and $\fm$ the set of major and minor arcs from Subsection~\ref{setupsub}. The contribution over the major arcs $\fM$ is estimated exactly as in the case of $\cR_{1,2}(N)$ in Subsection~\ref{majorsect}. However, here we note that
\begin{equation*}
    \cS_{1,1}(N):=\sum_{\substack{n_1,n_2\leq N\\ n_1+n_2=N\\(\rev{n_2},b)=1}}1=\#\fB(N)+O(1).
\end{equation*}
In particular, we have
\begin{equation}\label{majorbinary}
    \int_\fM S(\alpha)\rev{S}_b(\alpha)e(-N\alpha)\mathrm{d}\alpha=\fS_2(N)\#\fB(N)+O_{b,B}\left(\frac{N}{(\log N)^B}\right).
\end{equation}
By choosing a suitable value of $B$ in terms of $A$, it thus follows from~\eqref{majorbinary} and the triangle inequality that
\begin{align}
    &\sum_{\substack{x\leq N\leq 2x\\2\mid N}}\left|\cR_{1,1}(N)-\fS_2(N)\#\fB(N)\right|^2\notag\\
    &\qquad\ll_{b,A}\sum_{\substack{x\leq N\leq 2x\\ 2\mid N}}\left|\int_{\fm}S(\alpha)\rev{S}_b(\alpha)e(-N\alpha)\mathrm{d}\alpha\right|^2+\frac{x^3}{(\log x)^{A+1}},\label{bigmeaneq}
\end{align}
noting that $(a+b)^2\ll a^2+b^2$ for all real $a,b$. To bound the first term in~\eqref{bigmeaneq}, we use Bessel's inequality, which gives
\begin{align*}
    \sum_{\substack{x\leq N\leq 2x\\ 2\mid N}}\left|\int_{\fm}S(\alpha)\rev{S}_b(\alpha)e(-N\alpha)\mathrm{d}\alpha\right|^2\leq \int_{\fm}|S(\alpha)|^2|\rev{S}_b(\alpha)|^2\mathrm{d}\alpha.
\end{align*}
To finish, we apply Vinogradov's bound (Lemma~\ref{vinlem}) and Lemma~\ref{parslem}, giving
\begin{equation*}
    \int_{\fm}|S(\alpha)|^2|\rev{S}_b(\alpha)|^2\mathrm{d}\alpha\ll_{b,A}\frac{x^3}{(\log x)^A}.
\end{equation*}
This completes the proof of~\eqref{meanvalueeq} and thus Theorem~\ref{almostthm:intro}.

\section{Proof of Theorem~\ref{sqfreethm}}\label{sqfreesect}
In this section, we use Theorem~\ref{thm: smooth zsiflaw legeis} to prove Theorem~\ref{sqfreethm}. That is, we prove an asymptotic for the number of representations of an integer as the sum of a reversed prime and a square-free number.
\begin{proof}[Proof of Theorem~\ref{sqfreethm}]
    Using the standard identity $\mu^2(n)=\sum_{d^2\mid n}\mu(d)$, we have
    \begin{align*}
        \cR_{\square}(N)=\sum_{\substack{\rev{p}<N\\ (\rev{p},b^3-b)=1}}\mu^2(N-\rev{p})\log p&=\sum_{\substack{\rev{p}<N\\(\rev{p},b^3-b)=1}}\sum_{d^2\mid N-\rev{p}}\mu(d)\log p\\
        &=\sum_{d\leq \sqrt{N}}\mu(d)\sum_{\substack{\rev{p}<N\\(\rev{p},b^3-b)=1\\\rev{p}\equiv N\thinspace(d^2)}}\log p.
    \end{align*}
    We now split our argument based on the range of $d$, so that
    \begin{equation*}
        \cR_{\square}(N)=s_1(N)+s_2(N),
    \end{equation*}
    where 
    \begin{align*}
        s_1(N)&:=\sum_{d\leq (\log N)^{A+1}}\mu(d)\sum_{\substack{\rev{p}<N\\(\rev{p},b^3-b)=1\\\rev{p}\equiv N\thinspace(d^2)}}\log p,\\
        s_2(N)&:=\sum_{(\log N)^{A+1}<d\leq\sqrt{N}}\mu(d)\sum_{\substack{\rev{p}<N\\(\rev{p},b^3-b)=1\\\rev{p}\equiv N\thinspace(d^2)}}\log p.
    \end{align*}
    The sum $s_2(N)$ is readily bounded as
    \begin{align}\label{s2bound}
        |s_2(N)|&\leq \sum_{(\log N)^{A+1}<d\leq\sqrt{N}}\sum_{\substack{\rev{p}<N\\(\rev{p},b^3-b)=1\\\rev{p}\equiv N\thinspace(d^2)}}\log p\notag\\
        &\ll_b \sum_{(\log N)^{A+1}<d\leq\sqrt{N}}\left(\frac{N}{d^2}+1\right)\log N\ll\frac{N}{(\log N)^{A}},
    \end{align}
    which is sufficiently small as to be absorbed into the error term of the final result~\eqref{squarerep}. We thus shift our attention to $s_1(N)$. Here we apply Theorem~\ref{thm: smooth zsiflaw legeis} to obtain
    \begin{align}\label{s1bound1}
        s_1(N)&=\#\mathfrak{B}(N)\sum_{d\leq(\log N)^{2A+1}}\mu(d)\frac{(d^2,b^3-b)}{\varphi((d^2,b^3-b))}\frac{\rho_b(N,d^2)}{d^2}+O_{b,A}\left(\frac{N}{(\log N)^A}\right).
    \end{align}
    We write the sum above as
    \begin{equation*}
        \sum_{d\leq(\log N)^{2A+1}}=\sum_{d=1}^{\infty}-\sum_{d>(\log N)^{2A+1}}.
    \end{equation*}
    Here we can use that $(d^2,b^3-b)=O_b(1)$ to bound the tail sum as 
    \begin{align*}
        \sum_{d>(\log N)^{A+1}}\mu(d)\frac{(d^2,b^3-b)}{\varphi((d^2,b^3-b))}\frac{\rho_b(N,d^2)}{d^2}&\ll_b\sum_{d>(\log N)^{A+1}}\frac{1}{d^2}\ll\frac{1}{(\log N)^{A+1}}.
    \end{align*}
    As a result, we may ``complete" the sum in~\eqref{s1bound1}, with
    \begin{equation}\label{s1bound2}
        s_1(N)=\#\mathfrak{B}(N)\sum_{d=1}^\infty\mu(d)\frac{(d^2,b^3-b)}{\varphi((d^2,b^3-b))}\frac{\rho_b(N,d^2)}{d^2}+O_{b,A}\left(\frac{N}{(\log N)^A}\right).
    \end{equation}
    Now, applying the standard identity $m/\varphi(m)=\prod_{p\mid m}(p/(p-1))$, and converting to an Euler product,
    \begin{align}
        \sum_{d=1}^\infty\mu(d)\frac{(d^2,b^3-b)}{\varphi((d^2,b^3-b))}\frac{\rho_b(N,d^2)}{d^2}
        &=\sum_{d=1}^\infty\frac{\mu(d)}{d^2}\rho_b(N,d^2)\prod_{p\mid (d,b^3-b)}\left(\frac{p}{p-1}\right)\notag\\
        &=\prod_{p\mid b^3-b}\left(1-\frac{\rho_b(N,p)}{p(p-1)}\right)\prod_{p\nmid b^3-b}\left(1-\frac{\rho_b(N,p)}{p^2}\right)\notag\\
        &=\prod_{\substack{p\mid b^3-b\\p\nmid N}}\left(1-\frac{1}{p^2-p}\right)\prod_{p\nmid b^3-b}\left(1-\frac{1}{p^2}\right).\label{b3bprod1}
    \end{align}
    Here,
    \begin{align}
        \prod_{p\nmid b^3-b}\left(1-\frac{1}{p^2}\right)&=\frac{\prod_{p\geq 2}\left(1-1/p^2\right)}{\prod_{p\mid b^3-b}(1-1/p^2)}\notag\\
        &=\sum_{d=1}^\infty\frac{\mu(d)}{d^2}\prod_{p\mid b^3-b}\frac{p^2}{p^2-1}\notag\\
        &=\frac{1}{\zeta(2)}\prod_{p\mid b^3-b}\left(1+\frac{1}{p^2-1}\right)\label{b3bprod2}.
    \end{align}
    Substituting~\eqref{b3bprod2} into~\eqref{b3bprod1} and then~\eqref{s1bound2} gives
    \begin{equation}\label{s1asym}
        s_1(N)=\frac{\#\mathfrak{B}(N)}{\zeta(2)}\prod_{\substack{p\mid b^3-b}}\left(1+\frac{1}{p^2-1}\right)\prod_{\substack{p\mid b^3-b\\ p\nmid N}}\left(1-\frac{1}{p^2-p}\right)+O_{b,A}\left(\frac{N}{(\log N)^A}\right).
    \end{equation}
    To conclude, we recall that our bound~\eqref{s2bound} for $s_2(N)$ was $O(N/(\log N)^A)$, so that the asymptotic expression for $\cR_{\square}(N)$ is the same as that for $s_1(N)$ in~\eqref{s1asym}.
\end{proof}

\section{Further discussion on sums of reversed primes}\label{sumsect}
In this section, we discuss representing integers $N>0$ as \emph{purely} the sum of reversed primes, rather than a combination of primes and reversed primes. That is, we consider writing $N$ in the form
\begin{equation}\label{purereveq}
    N=\rev{p_1}+\cdots+\rev{p_k},
\end{equation}
for some $k\geq 2$. More precisely, we analyse the behaviour of the counting function
\begin{equation*}
    \cR_{0,k}(N)=\sum_{\substack{\rev{p_1},\ldots,\rev{p_k}\leq N\\\rev{p_1}+\cdots+\rev{p_k}=N\\ (\rev{p_1}\cdots\rev{p_k},b^3-b)=1}}(\log p_1)\cdots(\log p_k).
\end{equation*}

\subsection{Circle method heuristics}
We can obtain a conjectural asymptotic for $\cR_{0,k}(N)$ by applying the circle method and only considering the major arc contribution. In particular, applying the circle method as in Section~\ref{ternsect} leads us to conjecture the following.

\begin{conjecture}\label{purecon}
    Let $b\geq 2$ and $k\geq 2$ be fixed. Then, for any $A>0$,
    \begin{equation}\label{r0kasym}
        \mathcal{R}_{0,k}(N)=\mathfrak{S}_k(N)\cS_{0,k}(N)+O_{b,A}\left(\frac{N^{k-1}}{(\log N)^A}\right),
    \end{equation}
    where
    \begin{equation}\label{singkeq}
        \mathfrak{S}_k(N):=\prod_{\substack{p\mid b^3-b\\ p\mid N}}\left(1-\left(\frac{-1}{p-1}\right)^{k-1}\right)\prod_{\substack{p\mid b^3-b\\ p\nmid N}}\left(1-\left(\frac{-1}{p-1}\right)^k\right)
    \end{equation}
    and
    \begin{equation}\label{S0kdef}
       \cS_{0,k}(N):=\sum_{\substack{n_1,\ldots,n_k\leq N\\n_1+\cdots+n_k=N\\(\rev{n_1}\cdots\rev{n_k},b)=1}}1. 
    \end{equation}
\end{conjecture}

From~\eqref{singkeq} we see that $\fS_k(N)>0$ if and only if $k$ and $N$ have the same parity. The term $\cS_{0,k}(N)$ is more complicated with its positivity depending on the choice of the base $b$. However, as a simple example, if $b$ is a prime then $(\rev{n},b)=1$ is satisfied for all $n\in\N$ and 
\begin{equation*}
    \cS_{0,k}(N)\asymp N^{k-1}\qquad (b\ \text{prime})
\end{equation*}
so that Conjecture~\ref{purecon} yields an asymptotic. This justifies Conjecture~\ref{primebasecon} mentioned in the introduction.

The key barrier to proving Conjecture~\ref{purecon} is a suitable estimate for the minor arcs in the circle method. The fact that there are no primes in the sum~\eqref{purereveq} means that we cannot apply Vinogradov's bound (Lemma~\ref{vinlem}) as in the proof of Theorems~\ref{ternaryhcabthm1} and~\ref{ternaryhcabthm2}. So, if one could prove the following ``reverse" version of Vinogradov's bound, then Conjecture~\ref{purecon} would follow for $k\geq 3$.

\begin{hypothesis}\label{vinhypo}
    Let $N$ be sufficiently large and $B\geq 1$. Suppose that $\alpha$ is a real number satisfying
    \begin{equation*}
        \left|\alpha-\frac{a}{q}\right|\leq\frac{1}{q^2}\quad\text{with}\quad (\log N)^B<q\leq\frac{N}{(\log N)^B}
    \end{equation*}
    and $(a,q)=1$. Then,
    \begin{equation*}
        \rev{S}_b(\alpha):=\sum_{\substack{\rev{p}\leq N\\(\rev{p},b^3-b)=1}}e(\rev{p}\alpha)\log p\ll_{b,B}\frac{N}{(\log N)^A},
    \end{equation*}
    for some $A=A(B)>0$ with $A\to\infty$ as $B\to\infty$.
\end{hypothesis}

\begin{theorem}\label{hypconthm}
    For $k\geq 3$, Conjecture~\ref{purecon} follows from Hypothesis~\ref{vinhypo}.
\end{theorem}

We do not give a proof of Theorem~\ref{hypconthm} here as it follows routinely from the arguments in Section~\ref{ternsect}. Similar to other binary problems, the case $k=2$ in Conjecture~\ref{purecon} cannot be obtained via Hypothesis~\ref{vinhypo}, and so this case appears completely out of reach.

We also remark that a result very similar to Hypothesis~\ref{vinhypo} is given in~\cite[Lemma~10.3]{bhowmik2024telhcirid}. However, the range of parameters is not flexible enough for our application.

\subsection{Bounds for the reverse Schnirelmann constant}
\emph{Schnirelmann's constant} is defined to be the smallest number $K$ such that every $N\geq 2$ can be written as the sum of at most $K$ primes. In line with Goldbach's conjecture, it is believed that $K=3$, yet only the bound $K\leq 4$ is known (see~\cite{helfgott2015ternary}).

In view of Conjecture~\ref{purecon}, we can also define a reverse Schnirelmann constant, or \emph{Nnamlerinhcs' constant}, to be the smallest $\rev{K}_b>1$ such that every large integer $N$ can be written as the sum of at most $\rev{K}_b$ reversed primes in base $b$. It is still not known whether $\rev{K}_b<\infty$.

\begin{conjecture}\label{nnamexistcon}
    For all $b\geq 2$, Nnamlerinhcs' constant exists. That is, $\rev{K}_b<\infty$.
\end{conjecture}

In the case where $b$ is a prime, Conjecture~\ref{purecon} implies that $\rev{K}_b=3$, with every large even number expressible as the sum of two reversed primes and every large odd number expressible as the sum of three reversed primes\footnote{For the representation of a large odd number $N$, one can just add a small odd reversed prime $\rev{p}$ to a representation of the even number $N-\rev{p}$ as the sum of two reversed primes.}.

However, even when $b=10$ complications arise. In particular, for any $n>1$ one always has
\begin{equation*}
    6\cdot 10^n\neq\rev{p_1}+\rev{p_2}\qquad (b=10)
\end{equation*}
for any choice of reversed primes $\rev{p_1}$ and $\rev{p_2}$. This is because in base 10, a reversed prime must always have $1,3,7$ or $9$ as its first digit. Or, in the context of Conjecture~\ref{purecon}, one has $\cS_{0,2}(6\cdot 10^n)=0$ when $b=10$. Therefore, it instead seems that one should have $\rev{K}_{10}=4$, with every odd number expressible as the sum of three reversed primes (Conjecture~\ref{b10con}), and every even number expressible as the sum of \emph{four} reversed primes.

In general, for any fixed $b\geq 2$, one can use a combinatorial argument to obtain bounds for $\rev{K}_{b}$, assuming it exists. Notably, $\rev{K}_b$ can attain arbitrarily large values.

\begin{theorem}\label{Kbithm}
    There exists an infinite, increasing sequence of bases $(b_i)_{i\geq 1}$ with
    \begin{equation}\label{Kbilowereq}
        \rev{K}_{b_i}\gg \log b_i.
    \end{equation}
    In particular, $\rev{K}_{b_i}\to\infty$ as $b_i\to\infty$.
\end{theorem}
\begin{proof}
    Let $p_i$ be the $i$-th prime number and
    \begin{equation}\label{bidef}
        b_i=\prod_{p\leq p_i}p
    \end{equation} 
    be the product of all primes up to $p_i$. All primes $p>b_i$ expressed in the base $b_i$ cannot end in any digit $2,3,4\dots, p_i$. So for all $L\geq 2$, there do not exist any base-$b_i$ reversed primes in the interval 
    \begin{equation*}
        [2 b_i^{L-1},(p_i+1)b_{i}^{L-1}].
    \end{equation*}
    Thus, if $N=(p_i+1)b_i^{L-1}$ then one requires a sum of at least $(p_i+1)/2$ reversed primes to represent $N$ in base $b_i$. In other words
    \begin{equation*}
        \rev{K}_{b_i}\geq\frac{p_i+1}{2}\gg p_i.
    \end{equation*}
    To finish, we note that $p_i\sim\log b_i$ by~\eqref{bidef} and the prime number theorem.
\end{proof}

Setting $b_i$ to be a primorial in the above proof is the ``natural" choice if one wants to maximise the value of $K_{b_i}$. As a result, we strengthen Conjecture~\ref{nnamexistcon} by suggesting that the there is a matching upper bound to~\eqref{Kbilowereq}.

\begin{conjecture}\label{Kbcon}
    For all $b\geq 2$, one has $\rev{K}_b\ll\log b$.
\end{conjecture}

It is likely that one could deduce Conjecture~\ref{Kbcon} from Conjecture~\ref{purecon}, or an even stronger (conditional) bound such as
\begin{equation*}
    \rev{K}_b\ll\omega(b)\log\omega(b),
\end{equation*}
with $\omega(b)$ denoting the number of distinct prime factors of $b$. However, we do not attempt to do so here.

\section*{Acknowledgements}
We thank Bryce Kerr, Igor Shparlinski and Yuta Suzuki for helpful discussions. We also thank Kevin Destagnol for (casually) suggesting this project during the second author's visit to Paris in 2025.

\printbibliography
\end{document}